\documentclass{amsart}

\usepackage{amsmath}
\usepackage{amsthm}
\usepackage{amsopn}
\usepackage{amssymb}
\usepackage[all]{xy}

\allowdisplaybreaks[1]

\newcommand{\ul}[1]{\underline{#1}}
\newcommand{\mc}[1]{\mathcal{#1}}
\newcommand{\mb}[1]{\mathbb{#1}}
\newcommand{\mr}[1]{\mathrm{#1}}
\newcommand{\mbf}[1]{\mathbf{#1}}
\newcommand{\mit}[1]{\mathit{#1}}
\newcommand{\mf}[1]{\mathfrak{#1}}
\newcommand{\abs}[1]{\left\lvert #1 \right\rvert}
\newcommand{\norm}[1]{\lVert #1 \rVert}
\newcommand{\bra}[1]{\langle #1 \rangle}
\newcommand{\br}[1]{\overline{#1}}

\newcommand{\td}[1]{\widetilde{#1}}

\newcommand{\ZZ}{\mathbb{Z}}
\newcommand{\RR}{\mathbb{R}}

\newcommand{\QQ}{\mathbb{Q}}

\newcommand{\FF}{\mathbb{F}}
\newcommand{\GG}{\mathbb{G}}
\newcommand{\MS}{\mathbb{S}}

\newcommand{\PP}{\mathbb{P}}
\newcommand{\TT}{\mathbb{T}}
\newcommand{\LL}{\mathbb{L}}
\newcommand{\DD}{\mathbb{D}}

\newcommand{\Alg}{\mathrm{Alg}}
\newcommand{\Mod}{\mathrm{Mod}}
\newcommand{\Id}{\mathrm{Id}}
\newcommand{\Sp}{\mathrm{Sp}}
\newcommand{\Top}{\mathrm{Top}}

 \newtheorem{thm}[equation]{Theorem}
 \newtheorem{cor}[equation]{Corollary}
 \newtheorem{lem}[equation]{Lemma}
 \newtheorem{prop}[equation]{Proposition}

\theoremstyle{definition}

 \newtheorem{defn}[equation]{Definition}

 \newtheorem{rmk}[equation]{Remark}
\newtheorem*{thm*}{Theorem}
\newtheorem*{cor*}{Corollary}
\newtheorem*{lem*}{Lemma}
\newtheorem*{prop*}{Proposition}
\newtheorem*{defn*}{Definition}
\newtheorem*{ex*}{Example}
\newtheorem*{exs*}{Examples}
\newtheorem*{rmk*}{Remark}
\newtheorem*{claim*}{Claim}

\numberwithin{equation}{section}
\numberwithin{figure}{section}
\DeclareMathOperator{\Ext}{Ext}
\DeclareMathOperator{\Tor}{Tor}
\DeclareMathOperator{\Hom}{Hom}

\DeclareMathOperator{\Map}{Map}
\DeclareMathOperator{\Ind}{Ind}
\DeclareMathOperator{\Res}{Res}
\DeclareMathOperator{\Tr}{Tr}
\DeclareMathOperator{\Prim}{Prim}
\DeclareMathOperator{\Spf}{Spf}
\DeclareMathOperator{\TAQ}{TAQ}

\DeclareMathOperator*{\holim}{holim}
\DeclareMathOperator*{\hocolim}{hocolim}

\DeclareMathOperator*{\Tot}{Tot}

\title[The Bousfield-Kuhn functor and $\TAQ$]{The Bousfield-Kuhn functor and Topological Andr\'e-Quillen cohomology}
\author{Mark Behrens and Charles Rezk}
\date{\today}

\begin{document}

\begin{abstract}
We construct a natural transformation from the Bousfield-Kuhn functor evaluated on a space to the Topological Andr\'e-Quillen cohomology of the $K(n)$-local Spanier-Whitehead dual of the space, and show that the map is an equivalence in the case where the space is a sphere.  This results in a method for computing unstable $v_n$-periodic homotopy groups of spheres from their Morava $E$-cohomology (as modules over the Dyer-Lashof algebra of Morava $E$-theory).  We relate the resulting algebraic computations to the algebraic geometry of isogenies between Lubin-Tate formal groups.
\end{abstract}

\maketitle

\tableofcontents

\section{Introduction}

Let $X$ be a simply connected space.  Quillen \cite{Quillen} introduced two algebraic models of its rational homotopy type: a rational differential graded Lie algebra $\mc{L}(X)$ and a rational cocommutative differential graded coalgebra $\mc{C}(X)$.  The rational homotopy groups of $X$ are given by the homology of $\mc{L}(X)$ and the rational homology is given by the homology of $\mc{C}(X)$.  The differential graded Lie algebra $\mc{L}(X)$ can be obtained by taking the derived primitives of the differential graded coalgebra $\mc{C}(X)$.  Sullivan \cite{Sullivan} reformulated this theory in concrete terms in the case where $X$ is of finite type.  In Sullivan's theory, one takes a minimal model $\Lambda(X)$ of the differential graded algebra given by the dual of $\mc{C}(X)$.  The underlying graded commutative algebra of the minimal model $\Lambda(X)$ is free, and the rational homotopy groups of $X$ are recovered from the dual of the indecomposables of $\Lambda(X)$:
\begin{equation}\label{eq:Sullivan}
 \pi_*(X)_\QQ \cong \left(Q\Lambda(X)\right)^\vee.
 \end{equation}
Thus the rational homotopy groups of a space can be computed by taking the dual of the derived indecomposables of a commutative algebra model of its rational cochains.

Work of Kriz \cite{Kriz}, Goerss \cite{Goerss}, Mandell \cite{Mandell}, and Dwyer-Hopkins \cite[Appendix~C]{Mandell} shows that the unstable $p$-adic homotopy type of a simply connected finite type space is similarly encoded in its $\bar{\FF}_p$-valued singular cochains.  Unfortunately, the $p$-adic analog of (\ref{eq:Sullivan}) fails: the space of derived indecomposables of the resulting $E_\infty$-algebra is contractible \cite[Thm.~3.4]{Mandellcochains}.

One can nevertheless ask if there are other localizations of unstable homotopy groups for which the analog of (\ref{eq:Sullivan}) holds.  Unstable chromatic homotopy theory \cite{Bousfield} suggests that the $p$-primary unstable $v_h$-periodic homotopy groups are a likely candidate, as rational homotopy is the $h = 0$ case of this theory.  Integral homotopy groups are assembled out of the $v_h$-periodic homotopy groups for all $h$ and $p$. The purpose of this paper is to show that an analog of (\ref{eq:Sullivan}) sometimes holds in the case of $v_h$-periodic homotopy groups.

To state our results precisely, we will need to establish some notation.  Fix a prime $p$ and a height $h \ge 1$.  Let $E$ denote the Morava $E$-theory spectrum $E_h$ associated to the Lubin-Tate universal deformation $\GG$ of the Honda height $h$ formal group $\GG_0/\FF_{p^h}$.  Let $\mf{m}$ denote the maximal ideal of $E_*$, and let $K$ denote the even periodic variant of Morava $K$-theory $K(h)$ with 
$$ K_* = E_*/\mf{m}. $$
We shall use $(-)_K$ to denote localization with respect to $K$.
In this paper, unless explicitly stated otherwise, $E_*(-)$ will denote completed $E$-homology
$$ E_* Y = \pi_* (E \wedge Y)_K.$$
Let $T = T(h)$ denote the telescope of a $v_h$-self map on a type $h$ complex.  Bousfield and Kuhn (see, for example, \cite{Kuhn}) constructed functors $\Phi_{T}$ (for each $h \ge 1$) from pointed spaces to spectra with the properties that
\begin{align*}
\Phi_{T}(\Omega^\infty Y) & \simeq Y_{T}, \\
\pi_* \Phi_{T}(X) & \cong v_h^{-1} \pi_* (X)^{\wedge}.
\end{align*}
Here, $v_h^{-1} \pi_*(X)^{\wedge}$ are the completed $v_h$-periodic homotopy groups of $X$.\footnote{These should not be confused with the ``uncompleted'' unstable $v_h$-periodic homotopy groups
studied by Bousfield, Davis, Mahowald, and others. These are given as the homotopy groups of the $n$th telescopic monochromatic layer of $\Phi_T(X)$ (see \cite{Kuhncalc}).}
In this paper, we will be interested in the $K$-local variant
$$ \Phi(X) := \Phi_T(X)_K. $$
If the telescope conjecture is true, then $\Phi = \Phi_T$, but this conjecture is widely believed to be false for $h \ge 2$ \cite{MahowaldRavenelShick}.  The homotopy groups
$$ \pi_* \Phi(X) $$
thus constitute a competing definition of ``completed unstable $v_h$-periodic homotopy groups'' which are potentially more computable than $\pi_* \Phi_T(X)$. 

In this paper we construct a natural transformation between functors from pointed spaces to $K$-local spectra
$$ c^{S_K}: \Phi(X) \rightarrow \TAQ_{S_K}(S_K^{X_+}) $$
(the ``comparison map'') which relates $\Phi(X)$ to the topological Andr\'e-Quillen cohomology of the augmented commutative $S_K$-algebra $S_K^{X_+}$.  This relates unstable $v_h$-periodic homotopy to the dual of the derived indecomposables of the $E_\infty$ algebra of cochains with values in the $K$-local sphere, which generalizes (\ref{eq:Sullivan}) in the case of $h = 0$.
Our main theorem (Theorem~\ref{thm:main}) states that the comparison map is an equivalence when $X = S^q$ for $q$ odd.  The case of $h = 1$ is closely related to the thesis of Jennifer French \cite{French}.

In general, there is a class of finite spaces for which the comparison map is an equivalence.  This is discussed further in \cite{BehrensRezk}, where such spaces are called ``$\Phi_{K(h)}$-good''.  Thus the main result of this paper shows odd dimensional spheres are $\Phi_{K(h)}$-good.  In \cite{BehrensRezk}, we show one can easily deduce from this that even dimensional spheres, the groups $SU(n)$, and the groups $Sp(n)$, are $\Phi_{K(h)}$-good.  Also, finite products of finite $\Phi_{K(h)}$-good spaces are $\Phi_{K(h)}$-good.  However, the work of Langsetmo and Stanley \cite{LangsetmoStanley} (in the case of $h = 1$) and Brantner and Heuts \cite{BrantnerHeuts} ($h$ arbitrary) shows that mod $p$ Moore spaces are not $\Phi_{K(h)}$-good.  Brantner and Heuts also show in \cite{BrantnerHeuts} that wedges of spheres of dimension greater than $1$ are not $\Phi_{K(h)}$-good.

By the work of Ching \cite{Ching}, $\TAQ_{S_K}(S_K^{X_+})$ has the structure of an algebra over the operad formed by the Goodwillie derivatives $\partial_*(\mr{Id})$ of the identity functor on pointed spaces.  As this operad is Koszul dual to the commutative operad in spectra, it may be regarded as a topological analog of the (shifted) Lie operad.  Thus, at least for the class of finite $\Phi_{K(h)}$-good spaces, one could regard $\TAQ_{S_K}(S_K^{X_+})$ as a Lie algebra model for the unstable $v_h$-periodic homotopy type of $X$, mimicking in higher heights Quillen's Lie algebra model $\mc{L}(X)$ for the rational homotopy type.  This idea has since been successfully pursued by Heuts \cite{Heuts}, and is outlined in \cite{BehrensRezk}.

We apply our main theorem to understand the $v_h$-periodic Goodwillie
tower of the identity evaluated on odd spheres.  This constitutes a step in the program begun
by Arone and Mahowald \cite{AroneMahowald} to generalize the
Mahowald-Thompson approach to unstable $v_1$-periodic homotopy groups
of spheres \cite{Mahowald}, \cite{Thompson}.  Work of Arone-Mahowald
\cite{AroneMahowald} and Arone-Dwyer \cite{AroneDwyer} shows that
applying $\Phi$ to the Goodwillie tower of the identity evaluated on
$S^q$ ($q$ odd) gives a (finite) resolution 
\begin{equation}\label{eq:res}
\Phi(S^q) \rightarrow (L(0)_q)_K \rightarrow (L(1)_q)_K \rightarrow (L(2)_q)_K \rightarrow \cdots \rightarrow (L(h)_{q})_K.
\end{equation}
Here $L(k)_q$ denotes the Steinberg summand of the Thom spectrum of
$q$ copies of the reduced regular representation of $(\ZZ/p)^k$, as
described in \S\ref{sec:enlk}. 

We show (Theorem~\ref{thm:koszul}) that the $E$-homology of the
resolution (\ref{eq:res}) is isomorphic to the dual of the Koszul
resolution of the (degree $q$) Dyer-Lashof algebra $\Delta^q$ for
Morava $E$-theory (\ref{eq:deltaq}).  This results (Corollary~\ref{cor:koszul}) in a spectral sequence having the form
\begin{equation}\label{eq:EGSS}
\Ext^s_{\Delta^q}(\td{E}^q(S^q), \bar{E}_t) \Rightarrow E_{q+t-s} \Phi(S^q).
\end{equation}
This is related to unstable $v_h$-periodic homotopy groups of spheres by the homotopy fixed point spectral sequence \cite{DevinatzHopkins}
$$ H_c^s(\MS_h; E_t \Phi(S^q))^{\mathit{Gal}} \Rightarrow \pi_{t-s} \Phi(S^q). $$

In \cite{RezkMIC}, the second author defined the \emph{modular isogeny complex}, a cochain complex geometrically defined in terms of finite subgroups of the formal group $\GG$, mimicking the structure of the building for $GL_h(\FF_p)$.  We show that the cohomology of the modular isogeny complex is the dual of the Koszul resolution for $\Delta^q$, and use this to give a modular description of the differentials in the Koszul resolution.  This gives a modular interpretation of the $E_2$-term of the spectral sequence (\ref{eq:EGSS}). Presumably this modular interpretation is related to the ``pile'' interpretation of unstable chromatic homotopy, proposed by Ando, Morava, Salch, and others, but we do not pursue this here (see \cite{RezkICM}).

The results of this paper were first announced in 2012, however it took many years to resolve some thorny technical issues which emerged, the most significant of which involved the structure of the $E$-cohomology of $QX$ as an algebra over the Morava $E$-theory Dyer-Lashof algebra.  In fact, the first preprint version of this paper made an erroneous claim about this structure, and the authors are indebted to Nick Kuhn for discovering this error.  

During the evolution of the present form of this paper, many interesting developments have emerged.  The first author's Ph.D. student Guozhen Wang used some of the techniques of this paper to give a complete computation of $\pi_* \Phi(S^3)$ for $h = 2$ and $p \ge 5$ \cite{Wang}.  Yifei Zhu \cite{Zhu} has used our techniques to compute $E_*\Phi(S^{q})$ for $q$ odd and $h = 2$.  Finally, Arone-Ching \cite{AroneChing} and Heuts \cite{Heuts} have recently announced alternative approaches to prove of Theorem~\ref{thm:main} which are more conceptual than our computational approach, and more general, in the sense that they apply to the functor $\Phi_T$ as well as the functor $\Phi$. Both of these alternative approaches are are outlined in \cite{BehrensRezk}.

\subsection*{Organization of the paper.}$\quad$

In Section~\ref{sec:recollections} we summarize the results about the Morava $E$-theory Dyer-Lashof algebra $\Delta^q$ we will need for the rest of the paper.  

In Section~\ref{sec:BarrBeck} we introduce a form of Andr\'e-Quillen homology for unstable algebras over $\Delta^q$, as well as a Grothendieck-type spectral sequence which relates this homology to $\Tor_*^{\Delta^q}$.  

In Section~\ref{sec:TAQ} we introduce a bar construction model for Kuhn's filtration on topological Andr\'e-Quillen homology.  The layers of this filtration, as well as the layers of the Goodwillie tower of the identity, are equivalent to the spectra $L(k)_q$.  We also construct a Morava $K$-theory analog of a spectral sequence of Basterra, which relates the Morava $K$-homology of topological Andr\'e-Quillen homology to the algebraic Andr\'e-Quillen homology groups introduced in Section~\ref{sec:BarrBeck}.  

In Section~\ref{sec:enlk} we show the $E$-homology of the spectrum $L(k)_q$ is dual to the $k$-th term of the Koszul resolution for $\Delta^q$.  

In Section~\ref{sec:comparison}, we define the comparison map, and investigate its behavior on infinite loop spaces.  
This requires a technical result on the structure of the $E$-cohomology of $QX$ as a $\Delta^*$-algebra, which is relegated to
Appendix~\ref{apx:norm}.
  
In Section~\ref{sec:Weiss}, we discuss a $K$-local analog of Weiss's orthogonal calculus.  

In Section~\ref{sec:main}, we prove that the comparison map is an equivalence on odd spheres, by using $K$-local Weiss calculus to play the Goodwillie tower off of the Kuhn filtration.  

In Section~\ref{sec:kinv} we use the identification of the Goodwillie tower with the Kuhn filtration to compute the $E$-homology of the $k$-invariants of the Goodwillie tower.  From these results we establish the spectral sequence (\ref{eq:EGSS}).  

In Section~\ref{sec:modular} we give our modular description of the Koszul resolution for $\Delta^q$, by showing that it is given by the cohomology of the modular isogeny complex.  

There are two appendices.  Appendix~\ref{apx:borel} contains a summary of results on norms, transfers, and Euler classes.  These are needed in Appendix~\ref{apx:norm}, which is devoted to a detailed study of the Morava $E$-cohomology of $QX$.

\subsection*{Conventions.}$\quad$
\begin{itemize}
\item $(-)^\vee$ denotes the $E_0$-linear dual when applied to an $E_0$-module, and the Spanier-Whitehead dual when applied to a spectrum. 

\item $\Sp$ denotes the category of symmetric spectra with the positive stable model structure \cite{MMSS}, and shall simply refer to these as ``spectra''.  

\item If $R$ is a commutative $S$-algebra, $A$ is a commutative augmented $R$-algebra, and $M$ is an $R$-module, we will let $\TAQ^R(A;M)$ denote topological Andr\'e-Quillen homology of $A$ (relative to $R$) with coefficients in $M$.  Similarly we let $\TAQ_R(A;M)$ denote the corresponding topological Andr\'e-Quillen cohomology.  If $M = R$, we shall omit it from the $\TAQ$-notation.  

\item For an endofunctor $F$ of $\Top_*$, we shall let $\{ P_n(F)\}$ denote its Goodwillie tower, with fibers $D_n(F) = \Omega^\infty \mb{D}_n(F)$ and derivatives $\partial_n(F)$.
\end{itemize}

\subsection*{Acknowledgments}  The first author learned many of the techniques employed in this paper through conversations with his Ph.D. student Jennifer French, and also benefited from many conversations with Jacob Lurie.  Mike Hopkins suggested a version of the comparison map.  In some sense much of this paper completes a project initially suggested and pursued by Matthew Ando, Paul Goerss, Mike Hopkins, and Neil Strickland, relating $\Delta^q$ to the Tits building for $GL_h(\FF_p)$.  Johann Sigurdsson and Neil Strickland have also studied the Morava $E$-homology of $L(k)$, but from a slightly different perspective than taken in this paper.  The authors are indebted to Nick Kuhn, for pointing out a critical error in an earlier version of this paper, and to Greg Arone, Michael Ching, and Gijs Heuts, for generously sharing their unique insights on the results of this paper.  The first author would also like to thank Jack Morava, whose enthusiasm about this project was an invaluable source of motivation when technical obstacles began to multiply.  Finally, the authors would like to thank the referee, for their valuable comments.  The first author was supported by NSF grants  DMS-1050466, DMS-1452111, and DMS-1611786, and the second author was supported by NSF grants DMS-1006054 and DMS-1406121.

\section{Recollections on the Dyer-Lashof algebra for Morava $E$-theory}\label{sec:recollections}

\subsection*{Morava $E$-theory of symmetric groups}

Strickland studied the Hopf ring
$$ E^0 (\amalg_n B\Sigma_n),  $$
where the two products $\cdot$ and $\ast$ are given respectively by the cup product and transfers associated to the inclusions
\begin{equation}\label{eq:partsub}
\Sigma_n \times \Sigma_m \rightarrow \Sigma_{n+m} 
\end{equation}
and the coproduct is given by the restrictions associated to the above inclusions.  Note that there are actually inclusions 
(\ref{eq:partsub}) for every partition of the set $\{1, \ldots, n\}$ into two pieces.  We shall refer to the stabilizers of such partitions as \emph{partition subgroups}.

Strickland \cite{Strickland} proved that $E^0 (\amalg_n B\Sigma_n)$ is a formal power series ring (with respect to the $\ast$ product) with indecomposables
$$ \prod_{k \ge 0} E^0(B\Sigma_{p^k})/\Tr(\text{proper partition subgroups}). $$
Let 
$$ \mr{Sub}_{p^k}(\GG) = \Spf(\mc{S}_{p^k}) $$
 be the (affine) formal scheme of subgroups of $\GG$ of order $p^k$.  
For a Noetherian complete local $E_0$-algebra $R$, the $R$-points of $\mr{Sub}_{p^k}(\GG)$ are given by
$$ \mr{Sub}_{p^k}(\GG)(R) = \{ H < \GG \times_{\Spf(E_0)} \Spf(R) \: : \:  \abs{H} = p^k \}. $$
Strickland also shows that there is a canonical isomorphism
\begin{equation}\label{eq:Strickland}
E^0(B\Sigma_{p^k})/\Tr(\text{proper partition subgroups}) \cong \mc{S}_{p^k}
\end{equation}
of rings, where the product on the left-hand-side is induced from the $\cdot$ product. 

Let
$$ s: E_0 \rightarrow \mc{S}_{p^k} $$
be the map  induced topologically from the map
$$ B\Sigma_{p^k} \rightarrow \ast, $$  
which gives $\mc{S}_{p^k}$ an $E_0$-algebra-structure.
We regard $\mc{S}_{p^k}$ as a left module over $E_0$ by the module structure induced by $s$.  With respect this module structure, $\mc{S}_{p^k}$ is free of finite rank \cite[Prop.~6.3]{RezkWilk}.

Give the ring $\mc{S}_{p^k}$ the structure of a right $E_0$-module via the ring map
\begin{equation}\label{eq:t}
t : E_0 \rightarrow \mc{S}_{p^k}
\end{equation}
which associates to an $R$-point $H < \GG \times_{\Spf(E_0)} \Spf(R)$ the deformation 
$$ (\GG \times_{\Spf(E_0)} \Spf(R))/H. $$ 
The map $t$ arises topologically from the total power operation
$$ E^0(\ast) \rightarrow E^0(B\Sigma_{p^k}) $$
coming from the $E_\infty$ structure of $E$ \cite[Sec.~12.4]{AndoHopkinsStrickland}.

\subsection*{Morava $E$-theory of extended powers}

For an $E$-module $Y$, define
$$ \PP_E(Y) := \bigvee_{i \ge 0} Y^{\wedge_E i}_{h\Sigma_i}. $$
In \cite[\S4]{RezkWilk}, the second author defined a monad
$$ \TT: \mr{Mod}_{E_*} \rightarrow \mr{Mod}_{E_*} $$
and a natural transformation
$$ \TT \pi_*Y \rightarrow \pi_* (\PP_E Y)_{K} $$
which induces an isomorphism 
\begin{equation}\label{eq:algmodel}
[\TT \pi_*(Y)]^\wedge_{\mathfrak{m}} \xrightarrow{\cong} \pi_*(\PP_E Y)_{K}
\end{equation}
% MJB: I added a label to the above equation
if $\pi_*Y$ is flat as an $E_*$-module \cite[Prop.\ 4.9]{RezkWilk}.
% MJB: I removed the (unnecessary) finiteness hypothesis above.
% Here, $\mathfrak{m}$ denotes the maximal ideal of $E_0$.
There is a decomposition
$$ \TT = \bigoplus_{i \ge 0} \TT\bra{i} $$
so that if $\pi_*Y$ is finite and flat, we have
% MJB: added finiteness hypothesis above.
\begin{equation}\label{eq:extendedpower}
\pi_* [Y^{\wedge_E i}_{h\Sigma_i}]_K \cong \TT\bra{i} \pi_*Y_K. 
\end{equation}
% MJB: I fixed the above equation - it was a mess of a typo before
The monad $\TT$ comes equipped \cite[Prop.\ 4.7]{RezkWilk} with
natural isomorphisms 
\begin{equation}\label{eq:exponential}
\TT(M) \otimes_{E_*} \TT(N) \xrightarrow{\cong} \TT(M \oplus N).
\end{equation}
% MJB: I changed the E_0 to an E_* above because I forgot this was the graded context.
In particular, if $A$ is a $\TT$-algebra, then $A$ is a
graded-commutative $E_*$-algebra in the following strong sense: not
only do elements of odd degree anticommute, but also elements of odd
degree square to $0$  (see \cite[Prop.\ 3.4]{RezkWilk} for an
explanation of this phenomenon).

A convenient summary of the most important properties of the $\TT$ construction
is given in Section 3.2 of \cite{RezkKoszul}.  In particular, we note
that if $R$ is a $K(n)$-local commutative $E$-algebra, then $\pi_* R$
canonically admits the structure of a $\TT$-algebra.

\begin{lem}\label{lem:freealg}
If $M$ is a free $E_*$-module,  then $\TT M$ is a free graded
commutative $E_*$-algebra in the above sense.  
\end{lem}

\begin{proof}
  The rank 1 cases $M = E_*$ and $M = \Sigma E_*$ are discussed in the
  proof of Proposition~7.2 of \cite{RezkWilk}.  The general case then
  follows from \ref{eq:exponential}.
\end{proof}

Specializing to the case where $Y = \Sigma^q E$ (for $q \in \ZZ$), and
$i = p^k$, we have
$$ [\TT\bra{p^k} E_*(S^q)]_q = [E_* S^{qp^k}_{h\Sigma_{p^k}}]_q =
E_0(B\Sigma_{p^k})^{q\bar{\rho}_k} $$ 
where $\bar{\rho}_k$ denotes the reduced standard real representation
of $\Sigma_{p^k}$, and $(B\Sigma_{p^k})^{q\bar{\rho}_k}$ denotes the
associated Thom spectrum.

Consider the sub- and quotient modules
$$ \Prim_q[k] \hookrightarrow E_0(B\Sigma_{p^k})^{q\bar{\rho}_k} \twoheadrightarrow \Ind_q[k] $$
where $\Prim_q[k]$ denotes the intersection of the kernels of transfers to proper partition subgroups, and $\Ind_q[k]$ denotes quotient by the sum of the images of the restrictions from proper partition subgroups.  Both $\Prim_q[k]$ and $\Ind_q[k]$ are finite free $E_0$-modules, and (\ref{eq:Strickland}) implies that there is a canonical isomorphism
$$ \Prim_0[k] \cong \mc{S}^\vee_{p^k}.$$
Let $\iota_q$ denote the composite
$$ \iota_q: \Prim_q[k] \rightarrow \Ind_q[k]. $$
The suspension $\sigma$ is shown in \cite{RezkWilk} to fit these modules together to give a diagram
\begin{equation}\label{eq:zigzag}
\xymatrix@C+1em@R+1em{
& \Prim_{2m-1}[k] \ar[d]_\cong^{\iota_{2m-1}} & \Prim_{2m}[k] \ar@{^{(}->}[d]^{\iota_{2m}} & \Prim_{2m+1}[k] \ar[d]_\cong^{\iota_{2m+1}} & \cdots
\\
\cdots \ar[ur]^{\sigma}_\cong & \Ind_{2m-1}[k] \ar[ur]^\sigma_\cong & \Ind_{2m}[k] \ar[ur]^\sigma_{\cong} & \Ind_{2m+1}[k] \ar[ur]^\sigma_\cong 
}
\end{equation}
where $\iota_q$ is an isomorphism for $q$ odd, and an inclusion with torsion cokernel for $q$ even.  

% Combining Propositions 4.7 and 6.7 of \cite{RezkK}, we have the following proposition.

% \begin{prop}\label{prop:linearization}
% Suppose that $Y$ is a spectrum with $E_*Y$ finite and projective.
% Then there are isomorphisms
%\begin{align*}
%E_* Y^{\wedge n}_{h\Sigma_n}/\Res(\text{proper partition subgroups})  
%& 
%\cong \mc{L}(\TT\bra{n}) E_* Y 
%\\
%& \cong
%\begin{cases}
%\Ind_0[k] \otimes_{E_0} E_*Y, & n = p^k, \\
%0, & \text{otherwise}.
%\end{cases}
%\end{align*}
%Here, $\mc{L}(\TT\bra{n})$ denotes the linearization of the functor $\TT\bra{n}$.
%\end{prop}

\subsection*{The Dyer-Lashof algebra for Morava $E$-theory}

The algebra of additive power operations acting on cohomological degree $q$ is given by
$$ \Gamma^q = \bigoplus_k \Prim_{-q}[k]. $$
This is contained (via the map $\iota_{-q}$) in the larger algebra of indecomposable power operations
\begin{equation}\label{eq:deltaq}
 \Delta^q = \bigoplus_k \Ind_{-q}[k].
\end{equation}
In both rings, the ring $E_0$ is not central, and thus $\Gamma^q$ and $\Delta^q$ have distinct left and right $E_0$-module structures.  In the case of $\Gamma^0$, these left and right module structures are induced respectively from the left and right module structures of $E_0$ on $\mc{S}_{p^k}$ under the isomorphism
\begin{equation}\label{eq:dualGamma}
\Gamma^0[k] \cong \mc{S}^\vee_{p^k}. 
\end{equation}

The algebra $\Gamma^q$ is the algebra of natural endomorphisms of the functor
\begin{gather*}
U^q: \Alg_{\TT} \rightarrow \Mod_{E_0}, \\
A_* \mapsto A_{-q};
\end{gather*}
see \cite[\S3.8]{RezkKoszul}.  It follows that the underlying
$E_*$-module of a $\TT$-algebra carries the structure of a graded
$\Gamma^{*}$-module.  The morphism (\ref{eq:exponential}) gives this
$\Gamma^{*}$-module the structure of a graded-commutative
$\Gamma^*$-algebra.  The functors $U^q$ thus assemble to give a
functor
$$ U^*: \Alg_\TT \rightarrow \Alg_{\Gamma^*}. $$

The algebra $\Delta^q$ is the algebra of natural endomorphisms of the functor
\begin{gather}\label{eq:Vqdef}
\begin{split}V^q: \Alg_{\TT}\downarrow E_* \rightarrow \Mod_{E_0}, \\
A_* \mapsto [I(A)/I(A)^2]_{-q}.;
\end{split}
\end{gather}
see \cite[\S3.10]{RezkKoszul}.

The non-canonical natural isomorphisms $U^q \cong U^{q+2}$ and $V^q \cong V^{q+2}$ given by multiplication by a unit in $E_{-2}$ induce non-canonical isomorphisms of algebras
\begin{align}
\Gamma^q & \cong \Gamma^{q+2}, \label{eq:Gammacong} \\ 
\Delta^q & \cong \Delta^{q+2}.  \label{eq:Deltacong}
\end{align}
The suspension induces canonical isomorphisms of algebras
\begin{equation}\label{eq:susp}
 \sigma: \Delta^q \xrightarrow{\cong} \Gamma^{q-1}.
\end{equation}
In particular, all of the $E_0$-algebras $\Gamma^q$ and $\Delta^q$,
for all $q$, are non-canonically isomorphic to each other.  

For an $E_*$-module $M$ we define\footnote{Warning: the functor $\widehat{\TT}$ defined here is \emph{different} from the one defined by Barthel-Frankland \cite{BarthelFrankland}.}
\begin{equation}\label{eq:Thatdef}
\widehat{\TT} M = \prod_{i \ge 0} \TT\bra{i}M. 
\end{equation}
Note that by (\ref{eq:exponential}), Lemma~\ref{lem:freealg}, and (\ref{eq:deltaq}), for $M$ a free module over $E_*$, there is a non-canonical isomorphism 
\begin{equation}\label{eq:Tsym}
\TT M \cong \mr{Sym}_{E_*}(\Delta^* \otimes_{E_*} M).
\end{equation}
From this perspective, $\widehat{\TT}$ is non-canonically isomorphic to the corresponding formal power series ring
\begin{equation}\label{eq:That}
\widehat{\TT} M := \widehat{\mr{Sym}}_{E_*} (\Delta^* \otimes_{E_*} M). 
\end{equation}
Here $\widehat{\mr{Sym}}$ denotes the completed graded symmetric algebra.  For an $E$-module $Y$, let 
$$ \widehat{\PP}_E Y = \prod_{i \ge 0} Y^{\wedge_E i}_{h\Sigma_i} $$
denote the completed free commutative $E$-algebra.  The we have the following lemma.

\begin{lem}\label{eq:completealgmodel}
Suppose that $Y$ is an $E$-module, and that $\pi_*Y$ is flat over $E_*$.  Then (\ref{eq:algmodel}) induces an isomorphism
$$ [\widehat{\TT}\pi_*(Y)]^\wedge_\mf{m} \xrightarrow{\cong} \pi_* (\widehat{\PP}_E Y)_K. $$ 
\end{lem}

\begin{proof}
This follows from (\ref{eq:extendedpower}), together with the fact that $K$-localization commutes with products of $E$-local spectra \cite[Cor.~6.1.3]{BehrensDavis}, and \cite[Prop.~3.6-7]{RezkWilk}.
\end{proof}

\subsection*{The Koszul resolution}
Observe that the augmentation 
\begin{equation}\label{eq:augmentation}
 \epsilon: \Delta^q = \bigoplus_{k \ge 0} \Delta^q[k] \rightarrow \Delta^q[0] = E_0 
 \end{equation}
sending $\Delta^q[k]$ to $0$ for $k\geq0$ endows $E_0$ with the structure of a $\Delta^q$ bi-module: we shall
use the notation $\br{E}_0$ to denote this $\Delta^q$-bimodule.  Let
$\td{\Delta}^q$ denote the kernel of the augmentation $\epsilon$; it
is likewise a $\Delta^q$-bimodule.

Consider the normalized bar complex $B_*(\br{E}_0, \td{\Delta}^q, \br{E}_0)$ with
$$ B_s(\br{E}_0, \td{\Delta}^q, \br{E}_0) = \br{E}_0 \otimes_{E_0} (\td{\Delta}^q)^{\otimes_{E_0} s} \otimes_{E_0} \br{E}_0 \cong (\td{\Delta}^q)^{\otimes_{E_0} s}. $$
Endow the complex $B_*(\br{E}_0, \td{\Delta}^q, \br{E}_0)$ with a grading by setting
$$ B_s(\br{E}_0, \td{\Delta}^q, \br{E}_0)[k] :=  \bigoplus_{\substack{k=k_1+\cdots + k_s \\ k_i > 0}} \Delta^q[k_1] \otimes_{E_0} \cdots \otimes_{E_0} \Delta^q[k_s]. $$
Observe that since $\Delta^q[k]$ acts trivially on $\br{E}_0$ for $k > 0$, the differential in the bar complex preserves this grading.  Thus there is a decomposition of chain complexes
$$ B_*(\br{E}_0, \td{\Delta}^q, \br{E}_0) = \bigoplus_{k \ge 0} B_*(\br{E}_0, \td{\Delta}^q, \br{E}_0)[k]. $$

In \cite{RezkKoszul}, the second author proved that the algebras $\Delta^q$ are Koszul, as summarized in the following theorem.

\begin{thm}[\cite{RezkKoszul}, Prop.\ 4.6]\label{thm:RezkKoszul}
For each $k$, the $k$th graded part of the chain complex has homology
concentrated in top degree: 
$$ H_s(B_*(\br{E}_0, \td{\Delta}^q, \br{E}_0)[k]) = 
0, \qquad s \ne k.
$$
The top degree homology
$$ C[k]_{-q} := H_k(B_*(\br{E}_0, \td{\Delta}^q, \br{E}_0)[k])
$$
is finitely generated and free as a right $E_0$-module;
furthermore, $C[k]_{-q}=0$ if $k>h$.

Thus we have
\begin{align*}
\Tor^k_{\Delta^q}(\br{E}_0, \br{E}_0) & \cong C[k]_{-q}, \\
\Ext^k_{\Delta^q}(\br{E}_0, \br{E}_0) & \cong C[k]^\vee_{-q}.
\end{align*}
\end{thm}

\begin{rmk}
Actually, in \cite{RezkKoszul}, it is proven that $\Delta^0$ is Koszul, but using the isomorphisms (\ref{eq:Deltacong}) and (\ref{eq:susp}), there are non-canonical isomorphisms $\Delta^0 \cong \Delta^q$.  Therefore $\Delta^q$ is also Koszul.
\end{rmk}

If $M$ is a left $\Delta^q$-module, then the \emph{Koszul complex} $C_*^{\Delta^q}(M)$ associated to $M$ is the chain complex
$$ C_*^{\Delta^q}(M) = \left( C[0]_{-q} \otimes_{E_0} M \xleftarrow{\delta_0}  C[1]_{-q} \otimes_{E_0} M \xleftarrow{\delta_1} \cdots \right) $$
with differentials $\delta_k$ induced from the following diagram.
\begin{equation}\label{eq:Cdiff}
\xymatrix{
B_{k+1}(\br{E}_0, \Delta^q, M)[k+1] \ar@{=}[d] \ar[r]^-{d_{k+1}} & B_{k}(\br{E}_0, \Delta^q, M)[k] \ar@{=}[d] 
\\
\Delta^q[1]^{\otimes (k+1)} \otimes_{E_0} M   &
\Delta^q[1]^{\otimes k} \otimes_{E_0} M
\\
C[k+1]_{-q} \otimes_{\br{E}_0} M \ar[r]_{\delta_k} \ar@{^{(}->}[u] &
C[k]_{-q} \otimes_{\br{E}_0} M  \ar@{^{(}->}[u]
}
\end{equation}
Here, the map $d_{k+1}$ is the last face map in the bar complex $B_\bullet(\br{E}_0, \Delta^q, M)$.
We have
$$ H_s(C_*^{\Delta^q}(M)) \cong \Tor_{s}^{\Delta^q}(\br{E}_0, M). $$
Recall that the $E_0$-modules $C[k]_{-q}$ are projective.  It follows that if $M$ is projective as an $E_0$-module, the dual cochain complex computes $\Ext$:
$$ H^s(C_*^{\Delta^q}(M)^\vee) \cong \Ext^{s}_{\Delta^q}(M, \br{E}_0). $$

\section{Barr-Beck homology}\label{sec:BarrBeck}

\subsection*{Augmented $\TT$-algebras.}

Consider the adjunction
$$ \TT : \Mod_{E_*} \leftrightarrows \Alg_\TT: U^*. $$
A free $\TT$-algebra $\TT M$ is augmented over $E_*$ by the map
$$ \TT M \rightarrow \TT \bra{0}M = E_*. $$
Thus the above adjunction restricts to an adjunction for augmented $\TT$-algebras
$$ \TT : \Mod_{E_*} \leftrightarrows \Alg_\TT \downarrow E_*: I(-) $$
where $I(-)$ is the kernel of the augmentation.
The monad $\TT$ contains a ``non-unital'' summand 
$$ \bar{\TT} := \bigoplus_{i > 0} \TT\bra{i}. $$ 
Note that there is a natural isomorphism
$$ I(\TT M) \cong \bar{\TT} M. $$
In particular, if $A$ is an augmented $\TT$-algebra, then $I(A)$ is a
$\bar{\TT}$-algebra.

\subsection*{Trivial $\TT$-algebras}\label{sec:trivial-T-algebras}

The monad $\bar{\TT}$ is augmented over the identity functor via the projection
$$ \bar{\TT} \rightarrow \TT\bra{1} = \Id. $$
If $M$ is an $E_*$-module, then via the augmentation we can give $M$ the \emph{trivial} $\bar{\TT}$-algebra structure.  We shall denote the resulting $\bar{\TT}$-algebra by $\br{M}$.

\begin{rmk} It may appear that this notation conflicts in with the definition of the $\Delta^q$-module $\bar{E}_0$ immediately following Equation~(\ref{eq:augmentation}). Note, however, that in the case of $M = E_*$, the corresponding $\TT$-algebra $\bar{E}_*$ as defined above is a square-zero algebra, and hence is a $\Delta^*$-module.  The $\Delta^0$ module structures on $\bar{E}_0$ then agree in both cases.
\end{rmk}

If $X$ is an $E$-module spectrum, write $\br{X}$ for this spectrum
endowed with the structure of a non-unital $E$-algebra spectrum with
trivial multiplication.  We have the following.
\begin{prop}\label{prop:trivial-action}
If $X$ is a $K$-local $E$-module spectrum, the evident identification
$\pi_*\br{X} \approx \br{\pi_*X}$ is an isomorphism of
$\bar{\TT}$-algebras. 
\end{prop}

\subsection*{Cotriple homology.}

Suppose that we are given a functor $F: \Alg_\TT \downarrow E_* \rightarrow \mc{A}$ for $\mc{A}$ an abelian category.  Barr and Beck \cite{BarrBeck} define a ``cotriple homology'' associated to $F$ relative to the comonad $\TT I(-)$ on $\Alg_\TT \downarrow E_*$, which we shall simply denote $\LL_* F$,  as it could be viewed as a kind of left derived functor.  Explicitly it may be computed in terms of the monadic bar construction as
$$ \LL_s F(A) \cong H_s(F(B_*(\TT, \bar{\TT}, I(A)))). $$
% Similarly, given a contravariant functor $G: \Alg_\TT \downarrow E_* \rightarrow \mc{A}$, we get ``cotriple cohomology'' $\RR^*G$ with 
% $$ \RR^s G(A) \cong H^s(G(B_*(\TT, \bar{\TT}, I(A)))). $$

\subsection*{Derived functors of $\TT$-indecomposables.}

Let $N$ be an $E_0$-module.
Consider the functor
\begin{gather*}\label{eq:T-indecomposables}
\Omega^q_{\TT/E_*}(-;N) : \Alg_{\TT} \downarrow E_* \rightarrow \Mod_{E_0}, \\
 A \mapsto \br{N} \otimes_{\Delta^q} V^q(A).
% \\
% \mr{Der}^q_{\TT/E_*}(-, E_0): \Omega^q_{\TT/E_*} : \Alg_{\TT} \downarrow E_* \rightarrow \Mod_{E_0}, \\
 % A \mapsto \Hom_{E_0} (\Omega^q_{\TT/E_*} A, E_0)
\end{gather*}
where $V^q$ is the functor (\ref{eq:Vqdef}).

If $N = E_0$, we shall simply write
$$ \Omega^q_{\TT/E_*}A := \Omega^q_{\TT/E_*}(A;E_0). $$
Combining (\ref{eq:exponential}), Lemma~\ref{lem:freealg}, and the
definition of $\Delta^*$, we have the following lemma.

\begin{lem}\label{lem:T-indecomposable-of-free}
Suppose that $M$ is a free $E_*$-module.  Then there is a natural
isomorphism
\[
V^q(\TT M) \approx \Delta^q\otimes_{E_0} M_{-q},
\]
and hence a natural isomorphism
$$ \Omega^q_{\TT/E_*} (\TT M;N) \cong M_{-q} \otimes_{E_0} N. $$
\end{lem}

\begin{cor}
If $A \in \Alg_\TT \downarrow E_*$ is free as an $E_*$-module, then there is an isomorphism
\begin{align*}
\LL_s \Omega^q_{\TT/E_*} (A;N) & \cong H_{s}(B_*(\Id, \bar{\TT}, I(A))_{-q} \otimes_{E_0} N).
% \\
% \RR^s \Der^q_{\TT/E_*}(A, E_0) & \cong H^{s}(B_*(\Id, \bar{\TT}, I(A))^\vee_{-q}).
\end{align*}
\end{cor}

\subsection*{A Grothendieck spectral sequence.}

\begin{prop}\label{prop:GSS}
Suppose that $A$ is an augmented $\TT$-algebra which is free as an $E_*$-module.
Then there is a Grothendieck-type spectral sequence
\begin{align*}
E^2_{s,t} = \Tor^{\Delta^q}_s(\br{N}, \LL_t V^q(A)) & \Rightarrow \LL_{s+t} \Omega^q_{\TT/E_*} (A;N). 
% E_2^{s,t} = \Ext_{\Delta^q}^s(\LL_t V^q(A), \br{E}_0) & \Rightarrow \RR^{s+t} \Der^q_{\TT/E_*}(A, E_0).
\end{align*}
\end{prop}

\begin{proof}
% We construct the first spectral sequence, the second is handled in an analogous manner.
Consider the double complex
$$ C_{s,t} := B_s(\br{N}, \Delta^q, V^q(B_t(\TT, \bar{\TT}, I(A)))). $$
Computing the spectral sequence for the double complex by running $s$-homology, then $t$-homology, we have
$$ H_t H_s C_{s,t} \Rightarrow H_{s+t} \Tot C_{s,t}. $$
Using (\ref{eq:exponential}), Lemma~\ref{lem:freealg}, and the definition of $\Delta^q$, we have
\begin{align*} 
H_t H_s C_{s,t} & = 
H_t H_s B_s(\br{N}, \Delta^q, V^q(B_t(\TT, \bar{\TT}, I(A)))) \\
& \cong H_t \Tor^{\Delta^q}_{s}(\br{N}, V^q(B_t(\TT, \bar{\TT}, I(A)))) \\
& \cong H_t \Tor^{\Delta^q}_{s}(\br{N}, \Delta^q \otimes_{E_0} \bar{\TT}^{\circ t}I(A) ) \\
& \cong \begin{cases} H_t (B_t(\Id, \bar{\TT}, I(A))\otimes_{E_0} N), & s= 0, \\
0, & s \ne 0.  
\end{cases}
\end{align*}
The isomorphism of the second line uses the fact that
$V^q(B_t(\TT,\bar{\TT}, I(A)))$ is a free $E_0$-module when $A$ is
one, using Lemma~\ref{lem:freealg}.  The isomorphism of the third line
uses Lemma~\ref{lem:T-indecomposable-of-free}.

The spectral sequence therefore collapses to give an isomorphism
$$ H_{i} \Tot C_{*,*} \cong \LL_{i} \Omega^q_{\TT/E_*} (A;N). $$
Running $t$-homology followed by $s$-homology therefore gives a spectral sequence
$$ H_s H_t C_{s,t} \Rightarrow \LL_{s+t} \Omega^q_{\TT/E_*} (A;N).$$
Using the fact that $\Delta^q$ is free over $E_0$, we compute
\begin{align*}
H_s H_t C_{s,t} & = 
H_s H_t B_s(\br{N}, \Delta^q, V^q(B_t(\TT, \bar{\TT}, I(A)))) \\
& \cong H_s B_s(\br{N}, \Delta^q, H_t V^q(B_t(\TT, \bar{\TT}, I(A)))) \\
& \cong H_s B_s(\br{N}, \Delta^q, \LL_t V^q A) \\
& \cong \Tor^{\Delta^q}_s(\br{N}, \LL_t V^q A).
\end{align*}
\end{proof}

The homology groups $\LL_t V^q A$ appearing in the $E^2$-term of the
Grothendieck spectral sequence are demystified by the following
lemma.  We write ``$\LL_*\Omega_{(-)/E_*}$'' for the Andr\'e-Quillen
homology of augmented graded commutative $E_*$-algebras, where as in
\S\ref{sec:recollections}, graded commutativity implies that odd degree
elements square to $0$.  

\begin{lem}
Suppose that $A \in \Alg_\TT \downarrow E_*$ is free as an $E_*$-module.  Then there are isomorphisms
$$ \LL_i V^q A \cong [\LL_i \Omega_{A_*/E_*}]_{-q}. $$ 
\end{lem}

\begin{proof}
By Lemma~\ref{lem:freealg}, the bar resolution 
$$ B_\bullet (\TT, \bar{\TT}, I(A)) \rightarrow A $$
is a simplicial resolution of $A$ by free graded commutative algebras.  Since $V^*(-) = I(-)/I(-)^2$, the result follows.
\end{proof}

\begin{cor}\label{cor:edge}
Suppose that $A \in \Alg_\TT \downarrow  E_*$ is free as an augmented graded commutative $E_*$-algebra.  Then the Grothendieck spectral sequence collapses to give an isomorphism
\begin{align*}
\LL_s \Omega^q_{\TT/E_*} (A;N) & \cong \Tor^{\Delta^q}_{s}(\br{N}, V^q(A)). 
%\\
% \RR^s \Der^q_{\TT/E_*}(A, E_0) & \cong \Ext_{\Delta^q}^{s}(V^q(A), \bar{E}_0).
\end{align*}
\end{cor}

\subsection*{Linearization.}

The definition of $\Delta^*$ gives rise to natural transformations
\[
\Delta^*\otimes_{E_*} M \rightarrow V^*(\TT M) = \bar{\TT}(M)/(\bar{\TT}(M))^2
\leftarrow \bar{\TT}(M)
\]
of functors.
We have noted (Lemma~\ref{lem:T-indecomposable-of-free}) that if $M$ is
a free $E_*$-module, then $\Delta^*\otimes_{E*} M\rightarrow V^*(\TT M)$ is an
isomorphism.   Hence, on the full subcategory of free $E_*$-modules, we
obtain a natural transformation of monads
$$ \mc{L} : \bar{\TT} M \rightarrow \Delta^* \otimes_{E_*} M $$
on $\Mod_{E_*}$.  In \cite{RezkKoszul}, this transformation is
observed to be linearization for projective $M$.\footnote{We warn the reader that our notation here differs slightly from that used in \cite{RezkKoszul}: there the notation $\mc{L}_F$ is used for the linearization of a functor $F$, and $\epsilon: F \rightarrow \mc{L}_F$ is used for the natural transformation from a functor to its linearization.}
For $A \in \Alg_\TT \downarrow E_*$, the natural transformation $\mc{L}$ induces a map of chain complexes
\begin{equation}\label{eq:linearization}
\mc{L}: B(\Id, \bar{\TT}, I(A))_{-q} \rightarrow B_*(\br{E}_0, \Delta^q, V^q(A))
\end{equation}
and therefore a map
\begin{align*}
 \mc{L}: \LL_s \Omega^q_{\TT/E_*} A & \rightarrow \Tor^{\Delta^q}_{s}(\br{E}_0, V^q(A)).
% \\
 % \mc{L}^\vee: \RR^s \Der^q_{\TT/E_*} (A, E_0) & \leftarrow \Ext_{\Delta^q}^{s}(V^q(A), \br{E}_0 ).
\end{align*}

\begin{lem}\label{lem:quasi-isomorphism}
If $A$ is free as a graded commutative $E_*$-algebra, the map
(\ref{eq:linearization}) is a quasi-isomorphism. 
\end{lem}

\begin{proof}
This essentially follows Corollary~\ref{cor:edge} from an identification of the map (\ref{eq:linearization}) with the edge homomorphism of the Grothendieck spectral sequence.  Specifically, consider the following commutative diagram of maps of chain complexes.
\begin{equation}\label{eq:barcomplexes}
\xymatrix{
\bigoplus\limits_{s+t = n} B_s(\bar{E}_0, \Delta^q, V^q B_t(\TT, \bar{\TT}, I(A))) \ar[r]_-{\simeq}^-{\mr{aug}_t} \ar[d]^\simeq_{\mr{aug}_s} \ar[dr]_{\mc{L}} &
B_n(\bar{E}_0, \Delta^q, V^q A) 
\\
\bar{E}_0 \otimes_{\Delta^q} V^q(B_n(\TT, \bar{\TT}, I(A))) \ar@{=}[d] &
\bigoplus\limits_{s+t=n} B_s(\bar{E}_0, \Delta^q, B_t(\Delta^q, \Delta^q, V^qA)) \ar[d]_\simeq^{\mr{aug}_{s}} \ar[u]^\simeq_{\mr{aug}_t} 
\\
B_n(\Id, \bar{\TT}, I(A))_{-q} \ar[r]_{\mc{L}}  &
B_n(\bar{E}_0, \Delta^q, V^qA)
}
\end{equation}
Here the maps labeled $\mr{aug}_s$ and $\mr{aug}_t$ are the augmentations of the corresponding bar complexes, and $\mc{L}$ are the maps induced by linearization.
All of the augmentation maps are edge homomorphisms of appropriate spectral sequences of double complexes, with $E_2$-terms:
\begin{align*}
^{\mit{I}}E^2_{s,t} & = H_t H_s B_s(\bar{E}_0, \Delta^q, V^q B_t(\TT, \bar{\TT}, I(A))), \\
^{\mit{II}}E^2_{s,t} & = H_s H_t B_s(\bar{E}_0, \Delta^q, V^q B_t(\TT, \bar{\TT}, I(A))), \\
^{\mit{III}}E^2_{s,t} & = H_t H_s B_s(\bar{E}_0, \Delta^q, B_t(\Delta^q, \Delta^q, V^qA)), \\
^{\mit{IV}}E^2_{s,t} & = H_s H_t B_s(\bar{E}_0, \Delta^q, B_t(\Delta^q, \Delta^q, V^qA)).
\end{align*}
Each of these spectral sequences collapses:  the case of $\{^{\mit{I}}E^r_{s,t}\}$ is discussed in the proof of Proposition~\ref{prop:GSS}, the case of $\{^{\mit{II}}E^r_{s,t}\}$, the Grothendieck spectral sequence, is handled by Corollary~\ref{cor:edge}, and the spectral sequences $\{^{\mit{III}}E^2_{s,t} \}$ and $\{^{\mit{IV}}E^2_{s,t}\}$ collapse for trivial reasons.  It follows that each of the augmentation maps in Diagram~(\ref{eq:barcomplexes}) are quasi-isomorphisms, as indicated.  It follows that the bottom arrow in (\ref{eq:barcomplexes}) is a quasi-isomorphism, as desired.
\end{proof}

\section{Topological Andr\'e-Quillen homology}\label{sec:TAQ}

\subsection*{Definitions.}  Suppose that $R$ is a commutative $S$-algebra, and that $A$ is an augmented commutative $R$-algebra.
Topological Andr\'e-Quillen homology of $A$ (relative to $R$) was defined by Basterra \cite{Basterra} as a suitably derived version of the cofiber of the multiplication map on the augmentation ideal:
$$ \TAQ^R(A)  = I(A)/I(A)^{\wedge 2}. $$
If $M$ is an $R$-module, then topological Andr\'e-Quillen homology and cohomology of $A$ with coefficients in $M$ are defined respectively as
\begin{align*}
\TAQ^R(A;M) & = \TAQ^R(A) \wedge_R M, \\
\TAQ_R(A;M) & = F_R (\TAQ^R(A), M). \\
\end{align*}
As with $\TAQ^R$, we define $\TAQ_R(A) := \TAQ_R(A;R)$.

The augmentation ideal functor gives an equivalence
$$ I(-): \Alg_R \downarrow R \xrightarrow{\simeq} \Alg^{nu}_R $$
between the homotopy category of augmented commutative $R$-algebras and the category of non-unital commutative $R$-algebras \cite[Prop.~2.2]{Basterra}.  These categories are tensored over pointed spaces.  Basterra-McCarthy \cite{BasterraMcCarthy} show that $\TAQ^R(-)$ is the stabilization: there is an equivalence
\begin{equation}\label{eq:BasterraMcCarthy}
 \TAQ^R(A) \simeq \hocolim_n \Omega^n (S^n \otimes IA).
 \end{equation}

\subsection*{The Kuhn filtration.}

Kuhn \cite{KuhnMcCord} endows the topological Andr\'e-Quillen homology $\TAQ^S(A)$ of an augmented commutative $S$-algebra $A$ with an increasing filtration 
\begin{equation}\label{eq:Kuhnfilt}
F_1 \TAQ^S(A) \rightarrow F_2 \TAQ^S(A) \rightarrow \cdots. 
\end{equation}
% MJB: Made the following accurately reflect what we actually know.
We shall use the simplicial presentation of $\TAQ^S$ to give a point set level construction of Kuhn's filtration.  

\begin{rmk}
Our construction of the filtration is different from the construction given by Kuhn in \cite{KuhnMcCord}.  Kuhn and Pereira have recently explained to the authors an argument which shows that the two filtrations are equivalent \cite{KuhnPereira}.
\end{rmk}

Let $\PP$ denote the free $E_\infty$-ring monad on $\Sp$
$$ \PP(Y) := \bigvee_{n \ge 0} Y^{\wedge n}_{h\Sigma_n}, $$
and let $\td{\PP}$ denotes the ``non-unital'' version
$$ \td{\PP}(Y) := \bigvee_{n \ge 1} Y^{\wedge n}_{h\Sigma_n}. $$
Note that the monad $\td{\PP}$ is augmented over the identity.
Basterra \cite[\S5]{Basterra} shows that $\TAQ$ admits a simplicial presentation using the monadic bar construction:
\begin{equation}\label{eq:TAQbar}
 \TAQ^S(A) \simeq \abs{B_\bullet(\Id, \td{\PP}, I(A)) }.
\end{equation}

For a non-unital operad $\mc{O}$ in $\Sp$, let $\mc{F}_\mc{O}$ denote the free $\mc{O}$-algebra monad in $\Sp$:
$$ \mc{F}_\mc{O} Y := \bigvee_{n \ge 1} \mc{O}_n \wedge_{\Sigma_n} Y^{\wedge n}. $$
Let $\mr{Comm}$ denote the (non-unital) commutative operad in spectra, with
$$ \mr{Comm}_n = S. $$
Viewed as an endofunctor of spectra, we have (using \cite[Lem.~15.5]{MMSS}) 
$$ \LL \mc{F}_{\mr{Comm}} \simeq \td{\PP}. $$
We therefore have,  for $A$ positive cofibrant:
$$ \TAQ^S(A) \simeq \abs{B_\bullet(\Id, \mc{F}_{\mr{Comm}}, I(A)) }. $$
Observe for fixed $s$ there is a splitting 
\begin{equation}\label{eq:TAQsplitting}
\begin{split}
B_s(\Id, \mc{F}_{\mr{Comm}}, I(A)) & 
= \mc{F}_{\mr{Comm}}^s I(A) \\
& \cong \mc{F}_{[\mr{Comm}^{\circ s}]} I(A) \\
& = \bigvee_{i \ge 1} [\mr{Comm}^{\circ s}]_i \wedge_{\Sigma_i} I(A)^{\wedge i}. \\
& =: \bigvee_{i \ge 1} B_s(\Id, \mc{F}_{\mr{Comm}}, I(A))\bra{i}.
\end{split}
\end{equation}
Here, $\circ$ denotes the composition product of symmetric sequences.  Consider the filtration
\begin{align*}
F_n B_s(\Id, \mc{F}_{\mr{Comm}}, I(A)) & = \bigvee_{1 \le i \le n} B_s(\Id, \mc{F}_{\mr{Comm}}, I(A))\bra{i}. 
\end{align*}
This filtration is compatible with the simplicial structure, and therefore induces a filtration on realizations.  For positive cofibrant $R$, we have:
$$ F_n \TAQ^S(A) \simeq  \abs{F_n B_\bullet(\Id, \mc{F}_{\mr{Comm}}, I(A))}. $$

\subsection*{The layers of the Kuhn filtration}

We now compute the structure of the layers of the Kuhn filtration.  These layers have the same structure as that computed by Kuhn for the filtration he defined on $\TAQ$ in \cite{KuhnMcCord}, which seems to suggest the two filtrations agree.
The (pointed) partition poset complex $\mc{P}(n)_\bullet$ is defined to be the pointed simplicial $\Sigma_n$-set whose set of $s$-simplices is the set 
$$ \left\{ \lambda_0 \le \lambda_1 \le \cdots \le \lambda_{s} \: : \: 
\begin{array}{l} 
\text{$\lambda_i$ is a partition of $\ul{n}$,} \\
 \lambda_0 = \{ 1, \ldots, n \}, \\
\lambda_{s} = \{ 1 \} \cdots \{n\}
\end{array}
\right\} \amalg \{* \}.
$$
The face and degeneracy maps send the disjoint basepoint $*$ to the disjoint basepoint, and are given on the other elements by the formulas
\begin{align*}
d_i(\lambda_0 \le \cdots \le \lambda_s) & = 
\begin{cases}
\lambda_0 \le \cdots \le \widehat{\lambda}_i \le \cdots \le \lambda_s, & i \not\in \{ 0,s \}, \\
\ast, & i \in \{0, s\},
\end{cases}
\\
s_i(\lambda_0 \le \cdots \le \lambda_s) & = 
\lambda_0 \le \cdots \le \lambda_i \le \lambda_i  \le \cdots \le \lambda_s.
\end{align*}
Note that we have 
$$ 
\mc{P}(n)_0 = \begin{cases}
\{ \{1\}, \ast \}, & n = 1, \\
\{ \ast \}, & n > 1. 
\end{cases}
$$

\begin{prop}[c.f. \cite{KuhnMcCord}]
We have
\begin{equation}\label{eq:Kuhncofiber}
 F_n\TAQ^S(A)/F_{n-1}\TAQ^S(A) \simeq \abs{\mc{P}(n)_\bullet} \wedge_{h\Sigma_n} I(A)^{\wedge n}. 
\end{equation}
\end{prop}

\begin{proof}
Let $\br{I(A)}$ denote the spectrum $I(A)$ endowed with the trivial $\mc{F}_{\mr{Comm}}$-algebra structure.  
We have
\begin{align*}
 F_n\TAQ^S(A)/F_{n-1}\TAQ^S(A) 
& \simeq B(\Id, \mc{F}_{\mr{Comm}}, \br{I(A)})\bra{n} \\
& \cong \abs{[\mr{Comm}^{\circ \bullet}]_n \wedge_{\Sigma_n} I(A)^{\wedge n}} \\
& \cong \abs{B_\bullet(1, \mr{Comm}, 1)_n} \wedge_{\Sigma_n} I(A)^{\wedge n}. 
\end{align*}
Here, $1$ denotes the unit symmetric sequence.
The lemma now follows from the isomorphism of simplicial $\Sigma_n$-spectra (see \cite{Ching})
$$ B_\bullet(1, \mr{Comm}, 1)_n \cong \mc{P}(n)_\bullet. \qedhere$$
\end{proof}

\subsection*{A Basterra spectral sequence}

We conclude this section with the construction of a spectral sequence which computes the $K$-homology of $\TAQ$.

\begin{prop}\label{prop:BSS}
Suppose that $A$ is an augmented commutative $S_K$-algebra such that $E_*A$ is flat over $E_*$.  Then there is a spectral sequence
$$ E^2_{s,*} = \mb{L}_s \Omega^*_{\TT/E_*}(E_* A;K_0) \Rightarrow K_{s+*} \TAQ^{S_K}(A). $$
\end{prop}

\begin{proof}
We have
\begin{align*}
K \wedge \TAQ^{S_K}(A) & \simeq K \wedge \abs{B_\bullet(\Id, \td{\PP}_{S_K}, I(A))} \\
& \simeq \abs{K \wedge \td{\PP}^\bullet_{S_K} I(A)} \\
& \simeq \abs{K \wedge_E E \wedge \td{\PP}^\bullet_{S_K} I(A)} \\
& \simeq \abs{K \wedge_E \td{\PP}^\bullet_{E} (E \wedge I(A))}.
\end{align*}
Associated to this simplicial presentation is a Bousfield-Kan spectral sequence
\begin{equation}\label{eq:TAQSS}
E^1_{s,t} = \pi_t K \wedge_E \td{\PP}^s_{E} (E \wedge I(A)) \Rightarrow K_{s+t} \TAQ^{S_K}(A). 
\end{equation}
Using (\ref{eq:algmodel}), this $E^1$ term can therefore be described as
\begin{align*}
E^1_{s,*} & = \left( \bar{\TT}^s E_* I(A) \right)/\mf{m} \\
& = B_s(\mr{Id}, \bar{\TT}, E_* I(A)) \otimes_{E_0} K_0
\end{align*}
with $d_1$ differential given by the alternating sum of the face maps of the bar construction.
We therefore deduce that
$$ E^2_{s,*} = \mb{L}_s \Omega^*_{\TT/E_*}(E_* A;K_0). $$
\end{proof}

% \subsection*{A splitting lemma}

% The following lemma states that the Kuhn filtration splits for $A = S \rtimes Y$, the square-zero extension associated to a spectrum $Y$.

% \begin{lem}\label{lem:barpart}
% For a spectrum $Y$, the Kuhn filtration splits to give an equivalence
% $$
% \TAQ^S(S \rtimes Y) \simeq \bigvee_{n \ge 1} \TAQ^S(S \rtimes Y) \bra{n}
% $$ 
% where
% \begin{align*}
% \TAQ^S(S \rtimes Y)\bra{n} & = F_n \TAQ^S(S \rtimes Y)/F_{n-1} \TAQ^S(S \rtimes Y) \\
% & \cong \abs{\mc{P}(n)_\bullet} \wedge_{h\Sigma_n} Y^{\wedge n}. 
% \end{align*}
% \end{lem}

% \begin{proof}
% Note that $I(S \rtimes Y) = Y$, and the induced $\mc{F}_{\mr{Comm}}$-algebra structure on the spectrum $Y$ is encoded by the maps
% $$ Y^{\wedge n}/\Sigma_n \rightarrow Y$$
% which are the constant map $\ast$ for $n > 1$, and the identity map for $n = 1$.
% Because the terms of $\mc{F}_{\mr{Comm}}$ of degree greater than $1$ act trivially on $Y$, the splitting (\ref{eq:TAQsplitting}) respects the simplicial structure maps.  
% \end{proof}

\section{The Morava $E$-theory of $L(k)$}\label{sec:enlk}

\subsection*{$L(k)$-spectra}

Let $L(k)_q$ denote the spectrum given by \cite{Takayasu}
\begin{equation}\label{eq:L(k)def}
L(k)_q := \epsilon_{st} (B\FF_p^k)^{q\bar{\rho}_k}. 
\end{equation}
Here, $\bar{\rho}_k$ denotes the reduced regular real
representation of the elementary abelian $p$-group $\FF_p^k$, and
$(B\FF_p^k)^{q\bar{\rho}_k}$ denotes the Thom spectrum of the
$q$-fold direct sum of $\bar{\rho}_k$.  We write $\epsilon_{st}$ for
the Steinberg idempotent, acting on this spectrum, so that $L(k)_q$ is
the Steinberg summand.

Mitchell and Priddy \cite{MitchellPriddy} showed that there are
equivalences
$$ \Sp^{p^k}(S)/\Sp^{p^{k-1}}(S) \simeq \Sigma^k L(k)_1 $$
where $\Sp^{n}(S)$ is the $n$th symmetric product of the sphere spectrum.

The Goodwillie derivatives of the identity functor
$$ \Id: \Top_* \rightarrow \Top_* $$
are given by (see \cite{AroneMahowald})
\begin{equation}\label{eq:AroneMahowald}
 \partial_n(\Id) \simeq \left( \Sigma^\infty \abs{\mc{P}(n)_{\bullet}} \right)^\vee. 
\end{equation}

% MJB: Made the identification more succinct, correct, added remark
Arone and Dwyer \cite[Cor.\ 9.6]{AroneDwyer} establish mod $p$ equivalences (for $q$ odd)
\begin{align}
L(k)_q & \simeq_p  \Sigma^{k-q} [\partial_{p^k}(\Id) \wedge S^{qp^k}]_{h\Sigma_{p^k}}= 
\Sigma^{k-q} \DD_{p^k}(\Id)(S^q). \label{eq:ADI} % \\
% & \simeq  \Sigma^{-k-q}  [\partial_{p^k}(\Id)^\vee \wedge S^{qp^k}]_{h\Sigma_{p^k}}. \label{eq:ADII}
\end{align}
Here $\DD_{p^k}(\Id)$ is the infinite delooping of the $(p^k)$th layer of the Goodwillie tower of the identity functor on $\Top_*$.
 
\begin{rmk}
For the purposes of the rest of the paper, one could take (\ref{eq:ADI}) as the definition of the $p$-adic homotopy type of $L(k)_q$, instead of (\ref{eq:L(k)def}).  All of the computations and properties of the spectra $L(k)_q$ in what follows are really aspects of the partition poset model of $\DD_{p^k}(\Id)(S^q)$.
\end{rmk}
  
\subsection*{The $E$-homology calculation}

We now turn our attention to computing the $E$-homology of the spectra
$L(k)_q$ using (\ref{eq:ADI}).  We do this with a sequence of lemmas.
Recall from \S\ref{sec:trivial-T-algebras} that for an $E_*$-module $M$,
we write $\br{M}$ for the $\bar{\TT}$-algebra obtained by endowing $M$
with the trivial action.

\begin{lem}\label{lem:barpart}
If $Y$ is a spectrum with $E_* Y$ finite and flat as an $E_*$-module, then there is an isomorphism of simplicial $E_*$-modules
$$ 
E_* (\mc{P}(n)_\bullet \wedge_{h\Sigma_{n}} 
Y^{\wedge n}) \cong
B_\bullet(\Id, \bar{\TT}, \overline{E_*Y})\bra{n}. 
$$
\end{lem}

\begin{proof}
Replacing $Y$ with a cofibrant replacement in the positive model structure for symmetric spectra, 
this follows immediately from applying (\ref{eq:extendedpower}) to the isomorphisms
\begin{align*}
\mc{P}(n)_\bullet \wedge_{h\Sigma_{n}} 
Y^{\wedge n} & \cong B_\bullet(1, \mr{Comm}, 1)_n \wedge_{\Sigma_n} Y^{\wedge n} \\
& \cong B_\bullet(\Id, \mc{F}_{\mr{Comm}}, \br{Y})\bra{n}.
\end{align*}
\end{proof}

Recall from Theorem~\ref{thm:RezkKoszul} that $C[*]_q$ denotes the Koszul complex for $\Delta^{-q}$.

\begin{lem}\label{lem:calc1}
For $q$ odd, there is a canonical isomorphism
%\footnote{It might be nice to relate this stuff to Arone-Dwyer-Lesh.  The idea is that the functor $X\mapsto E_*(X\wedge_{h\Sigma_{p^k}} S^{qp^k})$ on finite $\Sigma_{p^k}$-sets is a Mackey functor $M_*$, and the $E_2$-term of the spectral sequence in the proof is $H_*^{\mathrm{Bre}}(\mc{P}(p^k)_\bullet;M_*)$.  I think their techniques show directly that this is concentrated in only one Bredon homology dimension, so the spectral sequence collapses.  Furthermore, they can identify the non-trivial term with the appropriate Steinberg piece.  ---C.}
$$ E_0 (\Sigma^{-k-q} \abs{\mc{P}(p^k)_\bullet} \wedge_{h\Sigma_{p^k}}
S^{qp^k}) \cong C[k]_q.$$
\end{lem}

\begin{proof}
Consider the Bousfield-Kan spectral sequence:
\begin{equation}\label{eq:BKSS}
E^1_{s,t} = E_{t}(\mc{P}(p^k)_s \wedge_{h\Sigma_{p^k}} S^{qp^k}) \Rightarrow E_{t+s} (\abs{\mc{P}(p^k)_\bullet} \wedge_{h\Sigma_{p^k}} S^{qp^k}). 
\end{equation}
We compute, using Lemma~\ref{lem:barpart} and Lemma~\ref{lem:quasi-isomorphism}
\begin{align*}
E_{q+*}(\mc{P}(p^k)_\bullet \wedge_{h\Sigma_{p^k}} S^{qp^k}) 
& \cong E_* \otimes_{E_0} B_\bullet(\Id, \bar{\TT}, \overline{E_* S^q})\bra{p^k}_{q} \\
& \xrightarrow[\simeq]{\mc{L}} E_* \otimes_{E_0} B_\bullet(\bar{E}_0, \Delta^{-q}, \bar{E}_0)[k].
\end{align*}
By Theorem~\ref{thm:RezkKoszul}, 
the spectral sequence (\ref{eq:BKSS}) collapses to give the desired result.
\end{proof}

\begin{rmk}
This can also be proven directly from the work of Arone, Dwyer, and Lesh 
\cite{AroneDwyerLesh}. 
\end{rmk}

% Using (\ref{eq:TAQsplitting}) and flat base change \cite[Prop.~4.6]{Basterra} we have the following corollary.

% \begin{cor}\label{cor:TAQH}
% For $q$ odd, we have
% $$ \pi_* \TAQ^E(E \rtimes \Sigma^q E) \cong \bigoplus_{k \ge 0} E_{*-k-q} \otimes_{E_0} C[k]_q. $$
% \end{cor}

% Using equivalence (\ref{eq:ADII}), we obtain our first description of the $E$-homology of $L(k)$:

% \begin{cor}
% For $q$ odd, we have
% $$ E_0 L(k)_q  \cong C[k]_q. $$
% \end{cor}

\begin{thm}\label{thm:enlk}
For $q$ odd, there are canonical isomorphisms of $E_*$-modules 
$$ E_0 L(k)_q  \cong C[k]^\vee_{-q} $$
and 
\[ E^0 L(k)_q \cong C[k]_{-q}. \]
\end{thm}

\begin{proof}
By (\ref{eq:ADI}) and (\ref{eq:AroneMahowald}) there are equivalences
\begin{align*}
L(k)_q & \simeq \Sigma^{k-q} \partial_{p^k}(\Id) \wedge_{h\Sigma_{p^k}} S^{qp^k} \\
& \simeq \Sigma^{k-q} \abs{\mc{P}(p^k)_\bullet}^\vee \wedge_{h\Sigma_{p^k}} S^{qp^k}.
\end{align*} 
Since $\mc{P}(p^k)_\bullet$ is a finite complex, the results of \cite{KuhnTate} imply that there are equivalences
\begin{align*}
\left([\Sigma^{k-q} \abs{\mc{P}(p^k)_\bullet}^\vee \wedge S^{qp^k}]_{h\Sigma_{p^k}}
\right)_K
& \xrightarrow[\mr{norm}]{\simeq} 
\left[ (\Sigma^{k-q} \abs{\mc{P}(p^k)_\bullet}^\vee \wedge S^{qp^k})_K \right]^{h\Sigma_{p^k}}
\\
& \simeq 
F(\Sigma^{-k+q} \abs{\mc{P}(p^k)_\bullet} \wedge S^{-qp^k}, S_K)^{h\Sigma_{p^k}} \\
& \simeq 
F((\Sigma^{-k+q} \abs{\mc{P}(p^k)_\bullet} \wedge S^{-qp^k})_{h\Sigma_{p^k}},S_K).
\end{align*}
Now apply the universal coefficient theorem, using the fact that $C[k]_{-q}$ is free as a module over $E_0$, to deduce the result from Lemma~\ref{lem:calc1}.  
\end{proof}

% \begin{rmk}
% An alternative to the last step in the previous proof (appealing to Theorem~\ref{thm:calc1}) is to actually run the dual version of the argument in the proof of Theorem~\ref{thm:calc1}.
% Consider the Bousfield-Kan spectral sequence:
% \begin{equation}\label{eq:BKSS2}
% E_1^{s,t} = E_{t}([ \mc{P}({p^k})_s \wedge_{h\Sigma_{p^k}} S^{-qp^k}]^\vee) \Rightarrow E_{t-s} ([
% \partial(\Id)^\vee_{p^k} \wedge_{h\Sigma_{p^k}} S^{-qp^k}]^\vee). 
% \end{equation}
% We compute, using Lemma~\ref{lem:barpart} and Lemma~\ref{lem:quasi-isomorphism}
% \begin{align*}
% E_{q+*}([\mc{P}(p^k)_s \wedge_{h\Sigma_{p^k}} S^{-qp^k}]^\vee) 
% & \cong E_* \otimes_{E_0} [B_\bullet(\Id, \bar{\TT}, \overline{E^* S^q})\bra{p^k}_{-q}]^\vee \\
% & \xleftarrow[\simeq]{\mc{L}^\vee} E_* \otimes_{E_0} B_\bullet(\bar{E}_0, \Delta^{q}, \bar{E}_0)[k]^\vee.
% end{align*}
% By Theorem~\ref{thm:RezkKoszul}, 
% the spectral sequence (\ref{eq:BKSS2}) collapses to give the desired result.  This perspective links this second approach to topological Andr\'e-Quillen \emph{cohomology}.
% \end{rmk}

% Using the dual of (\ref{eq:TAQsplitting}) and flat base change \cite[Prop.~4.6]{Basterra} ) (or by just dualizing Corollary~\ref{cor:TAQH}) we have the following corollary.

% \begin{cor}\label{cor:TAQC}
% For $q$ odd, we have
% $$ \pi_* \TAQ_E(E \rtimes \Sigma^{-q} E) \cong \prod_{k \ge 0} E_{*+k-q} \otimes_{E_0} C[k]^\vee_{-q}. $$
% \end{cor}

\begin{rmk}
Arone and Dwyer actually give another identification of the spectrum $L(k)_q$, dual to (\ref{eq:ADI}): they prove that there is an equivalence
$$ L(k)_q \simeq  \Sigma^{-k-q}  [\abs{\mc{P}(p^k)_\bullet} \wedge S^{qp^k}]_{h\Sigma_{p^k}}. 
$$
Thus Lemma~\ref{lem:calc1} gives the following alternative to Theorem~\ref{thm:enlk}: for $q$ odd we have
$$ E_0 L(k)_q  \cong C[k]_q. $$
This description of the $E$-homology of $L(k)_q$ is less well suited to the perspective of the present paper.
\end{rmk}

\section{The Bousfield-Kuhn functor and the comparison map}\label{sec:comparison}

\subsection*{The Bousfield-Kuhn functor.}

Let $T$ denote any $v_h$-telescope on a type $h$ finite complex.  
The Bousfield-Kuhn functor $\Phi_T$ factors localization with
respect to $T$.  We are mainly interested in the $K$-localization of 
$\Phi_T$, which we shall denote $\Phi$.
Thus we have a diagram of functors commuting up to natural weak equivalence.
$$
\xymatrix{
\Sp \ar[d]_{\Omega^\infty} \ar[r]^{(-)_{T}} \ar@/^2pc/[rr]^{(-)_K}
& \Sp \ar[r]^{(-)_{K}}
& \Sp
\\
\Top_* \ar[ur]^{\Phi_T}  \ar[urr]_{\Phi}
}
$$
The completed unstable $v_h$-periodic homotopy groups of $X$ are the homotopy groups of $\Phi_T(X)$:
$$ v_h^{-1}\pi_*(X) \cong \pi_* \Phi(X)^{\wedge}.  $$
If the telescope conjecture for height $h$ is true, then the functors $\Phi_T$ and  $\Phi$ are equivalent.  

See \cite{Kuhn} for a detailed summary of the construction and
properties of these functors.   The main additional property we will
need is that $\Phi$ commutes with finite homotopy limits \cite{Bousfield2}, and thus in
particular $\Phi\Omega\rightarrow \Omega\Phi$ is a natural weak equivalence.  

Applying $\Phi$ to the unit of the adjunction
$$ X \rightarrow \Omega^\infty \Sigma^\infty X, $$
we get a natural transformation
$$ \eta_X: \Phi(X) \rightarrow (\Sigma^\infty X)_{K}. $$

\subsection*{The comparison map.}

Let $R$ be a commutative $S$-algebra, and
consider the functor 
$$ R^{(-)_+}: \Top^{op}_* \rightarrow \Alg_R \downarrow R. $$
Here, the $R$-algebra structure on $R^{X_+}$ comes from the diagonal
on $X$, with unit given by the map $X \rightarrow \ast$, and
augmentation coming from the basepoint on $X$.

The augmentation ideal $I(R^{X_+})$ is identified with
$R^X$, the
$R$-module of maps from  $\Sigma^\infty X$ to
$R$.  As the functor $\Top^{op}_*\rightarrow \Alg_R^{nu}$ given by
$X\mapsto R^X$ is a pointed homotopy functor, 
there is are natural transformations 
$$ S^n \otimes  R^{X} \rightarrow R^{\Omega^n X}.  $$

Assume that $R$ is $K$-local.  We define a natural transformation
\[
c_R\colon \TAQ^R(R^{X_+}) \rightarrow R^{\Phi(X)}
\]
of functors $\Top_*^{op}\rightarrow \Mod_R$
as follows (using (\ref{eq:BasterraMcCarthy})):  
\begin{align*}
c_R: \TAQ^R(R^{X_+}) & \simeq \hocolim_n \Omega^n (S^n \otimes R^X) \\
& \rightarrow \hocolim_n \Omega^n R^{\Omega^n X} \\
& \simeq \hocolim_n \Omega^n R^{(\Sigma^\infty \Omega^n X)_K} \\
& \xrightarrow{\eta_{\Omega^n X}^*} \hocolim_n \Omega^n R^{\Phi (\Omega^n X)} \\
& \simeq \hocolim_n \Omega^n R^{\Sigma^{-n} \Phi (X)} \\
& \simeq R^{\Phi(X)}.
\end{align*}
Taking the $R$-linear dual of $c_R$ and composing with the evident map
$R\wedge \Phi(X) \rightarrow \Hom_R(R^{\Phi(X)}, R)$ gives a natural transformation
$$ c^R: (R \wedge \Phi(X))_{K} \rightarrow \TAQ_R(R^{X_+}). $$ 
We shall refer to $c_R$ and $c^R$ as the \emph{comparison maps}.

%If $(R \wedge \Phi(X))_K$ is dualizable as a $K$-local $R$-module, then the comparison map 
%$c^R$ is the $R$-linear dual of $c_R$.

\subsection*{The comparison map on infinite loop spaces.}

Let $Y$ be a spectrum.  The counit of the adjunction
$$ \epsilon: \Sigma^\infty \Omega^\infty \rightarrow \Id $$
induces a natural transformation
$$ \epsilon^* : S_K^Y \rightarrow S_K^{\Omega^\infty Y}. $$
Regarding $S_K^{\Omega^\infty Y}$ as a non-unital commutative $S_K$-algebra, this induces a map of augmented commutative $S_K$-algebras
$$ \td{\epsilon}^* : \PP_{S_K} S_K^{Y} \rightarrow S_K^{\Omega^\infty Y_+}. $$

The following property of $c_{S_K}\colon \TAQ^{S_K}(S_K^{\Omega^\infty
  Y_+})\rightarrow S_K^{\Phi(\Omega^\infty Y)}$ will be all that we need to know about it. 

\begin{lem}\label{lem:omegainfty}
The composite
$$ S_K^Y \simeq \TAQ^{S_K}(\PP_{S_K} S_K^Y) \xrightarrow{\TAQ^{S_K}(\td{\epsilon}^*)} 
\TAQ^{S_K}(S_K^{\Omega^\infty Y_+})
\xrightarrow{c_{S_K}} S_K^{\Phi(\Omega^\infty Y)} \simeq S_K^Y $$
is the identity.
\end{lem}

\begin{proof}
The lemma is proved by the commutativity of the following diagram
$$
\xymatrix{
S^Y 
\ar `l[dddd] `[dddd]_{\epsilon^*} [ddddr]
\ar@{^{(}->}[d] \ar[r]^-{\simeq} &
\TAQ^{S_K}(\PP_{S_K} S_K^{Y}) \ar[r]^{\td{\epsilon}^*} \ar@{=}[d] &
\TAQ^{S_K}(S_K^{\Omega^\infty Y_+}) \ar@{=}[d] 
\ar `r[dddd] `[dddd]^{c_{S_K}} [dddd] 
\\
\PP_{S_K} S_K^Y \ar[r] \ar[dr] \ar[dddr]_{\td{\epsilon}^*} &
\varinjlim \Omega^n(S^n \otimes \td{\PP}_{S_K} S_K^Y) \ar[r]_{\td{\epsilon}^*} \ar[d] &
\varinjlim \Omega^n(S^n \otimes S_K^{\Omega^\infty Y}) \ar[d]
\\
& \varinjlim \Omega^n(\td{\PP}_{S_K} S_K^{\Sigma^{-n} Y}) \ar[r]_{\td{\epsilon}^*} &
\varinjlim \Omega^n(S_K^{\Omega^\infty \Sigma^{-n} Y}) \ar[d]^{\eta_{\Omega^\infty \Sigma^{-n} Y}^*}
\\
&& \varinjlim \Omega^n(S_K^{\Sigma^{-n} Y}) 
\\
& S_K^{\Omega^\infty Y} \ar[ruu] \ar[r]_{\eta^*_{\Omega^\infty Y}} &
S_K^Y \ar[u]_{\simeq} 
}
$$
together with the fact that $\eta^*_{\Omega^\infty Y} \circ \epsilon^* \simeq \Id$ \cite[Sec.~7]{Kuhn}.
\end{proof}

\begin{rmk}
Since the the previous lemma implies that the comparison map $c^R$ is the inclusion of a wedge summand for $K$-locally dualizable infinite loop spaces $Y$, it provides an interesting mechanism for computing the effect of the Bousfield-Kuhn functors on maps between infinite loop spaces using $\TAQ$, even if we do not know the comparison map is an equivalence for these spaces (see \cite{Wangmono}, where this observation is employed to give a determination of the effect of the James-Hopf map on $E_*$-homology which is independent of Theorem~\ref{thm:modularkoszuldiffs}).  
\end{rmk}

\subsection*{The comparison map on $QX$.}

The previous lemma allows us to deduce the following.

%MJB: New version of this proposition 
\begin{prop}\label{prop:comparisonQ}
There is a non-negative integer $N$  such that for all pointed spaces of the form $X = \Sigma^N X'$ with finite free $E$-homology, the comparison map for $QX$
$$ c_{S_K} : \TAQ^{S_K}(S_K^{QX_+}) \rightarrow S_K^{\Phi QX} \simeq S_K^{X} $$
is an equivalence.
\end{prop}

\begin{proof}
We will argue that $\TAQ^{S_K}(S_K^{QX_+})$ has finite $K$-homology, of rank equal to the rank of the $K$-homology of  $S_K^{X}$.  
The proposition then follows from Lemma~\ref{lem:omegainfty}.  

We will use the Basterra spectral sequence (Prop.~\ref{prop:BSS}):
\begin{equation}\label{eq:BSS}
E^2_{s,*} = \mb{L}_s \Omega^*_{\TT/E_*}(E_* S_K^{QX_+};K_0) \Rightarrow K_{s+*} \TAQ^{S_K}(S_K^{QX_+}).
\end{equation}
The calculation of this $E^2$ term requires a thorough understanding of the $E$-homology of $S_K^{QX_+}$, as a $\TT$-algebra, modulo the maximal ideal $\mf{m}$.  This is the subject of Appendix~\ref{apx:norm}.  We shall freely refer to results in this appendix for the remainder of this proof.

We first note that (\ref{eq:SKQX}) and Corollary~\ref{cor:flat} imply that $E_*S_K^{QX_+}$ satisfies the flatness hypotheses required by Proposition~\ref{prop:BSS} for the $E_2$-term to take the desired form.
The main result of Appendix~\ref{apx:norm} is Lemma~\ref{lem:keylemma}, which implies that for $N$ sufficiently large, there is a simplicial isomorphism of bar constructions
$$ B_\bullet (\mr{Id}, \bar{\TT}, E_* S_K^{QX})/\mf{m} \cong B_\bullet (\mr{Id}, \bar{\TT}, \widehat{\TT} \td{E}^*X)/\mf{m} $$
where $\widehat{\TT}$ is the functor (\ref{eq:Thatdef}).
Therefore there is an isomorphism
$$
\mb{L}_s \Omega^*_{\TT/E_*}(E_* S_K^{QX_+};K_0) \cong
\mb{L}_s \Omega^*_{\TT/E_*}(\widehat{\TT} \td{E}^* X;K_0).
$$
The Grothendieck spectral sequence of Proposition~\ref{prop:GSS}
$$
E^2_{s,t} = \Tor^{\Delta^*}_s(\bar{K}_0, \LL_t V^*(\widehat{\TT} \td{E}^* X)) \Rightarrow \LL_{s+t} \Omega^*_{\TT/E_*} (\widehat{\TT} \td{E}^* X;K_0)
$$
collapses to give 
$$ \LL_s \Omega^*_{\TT/E_*}(\widehat{\TT} \td{E}^* X;K_0)
\cong 
\begin{cases}
\td{K}^* X,  & s = 0, \\
0, & s > 0.
\end{cases}
$$
We conclude that spectral sequence (\ref{eq:BSS}) converges and collapses to give an isomorphism
$$ K_* \TAQ^{S_K}(S^{QX_+}_K) \cong \td{K}^* X. $$
\end{proof}

\begin{rmk}
The authors would like to believe that Proposition~\ref{prop:comparisonQ} is an equivalence for all connected $X$ with finite free $E$-homology.
Ideally, some kind of weak convergence of the $K$-based cohomological Eilenberg-Moore spectral sequence for the path-loop fibration for $Q\Sigma^N X'$ would give a $K$-local equivalence
$$ S_K \wedge_{S_K^{Q\Sigma^{N} X'_+}} S_K \xrightarrow{\simeq} S_K^{Q\Sigma^{N-1} X'_+}. $$
It would then follow that there is a $K$-local equivalence
$$ \Sigma \TAQ^{S_K}(S_K^{Q\Sigma^{N} X'_+}) \simeq \TAQ^{S_K}(S_K^{Q\Sigma^{N-1}X'_+}). $$
The general result would then follow from downward induction on $N$.
\end{rmk}

\section{Weiss towers}\label{sec:Weiss}

In this section we freely use the language of Weiss's orthogonal calculus \cite{Weiss}.

\begin{defn}
Let $F$ be a reduced homotopy functor from complex vector spaces to $K$-local spectra.  We shall say that a tower
$$ \cdots \rightarrow F_n \rightarrow F_{n-1} \rightarrow \cdots \rightarrow F_1 $$
of functors under $F$ is a \emph{finite $K$-local Weiss tower} if 
\begin{enumerate}
\item the fiber of
$ F_n \rightarrow F_{n-1} $
is equivalent to the $K$-localization of a homogeneous degree $n$ functor from complex vector spaces to spectra, and 
\item The map $F \rightarrow F_n$ is an equivalence for $n \gg 0$.
\end{enumerate}
\end{defn}

\begin{rmk}
Suppose that $\{F_n \}$ is a finite $K$-local Weiss tower for $F$.  We record the following observations.
\begin{enumerate}
\item The functor $F_n$ is $n$-excisive.  This is because the localization of a homogeneous degree $n$ functor is $n$-excisive.  
\item If $\{G_n\}$ is a finite $K$-local Weiss tower for $G$, and $F \rightarrow G$ is a natural transformation, there is a homotopically unique induced compatible system of natural transformations
$$ F_n \rightarrow G_n. $$
This is because if $D_n$ is a homogeneous degree $n$ functor which is $K$-locally equivalent to the fiber $F_n \rightarrow F_{n-1}$, the space of natural transformations
$$ \mbf{Nat}((D_n)_K, G_m) \simeq \mbf{Nat}( D_n, G_m) $$
is contractible for $m < n$.  It follows that the natural map
$$ \mbf{Nat}(F_m, G_m) \xrightarrow{\simeq} \mbf{Nat}(F, G_m) $$
is an equivalence.
\item It follows from (2) that if $F$ admits a finite $K$-local Weiss tower, such a tower is homotopically unique.
\end{enumerate}
\end{rmk}

We will construct finite $K$-local Weiss towers of the following functors from complex vector spaces to spectra:
\begin{align*}
 V & \mapsto \Phi(\Sigma S^{V}), \\
 V & \mapsto \TAQ_{S_K}(S_K^{(\Sigma S^V)_+}).
\end{align*}
In each of these cases, the towers will only have non-trivial layers in degrees $p^k$ for $k \le h$.

\begin{prop}\label{prop:Weiss1}
The tower $\{ \Phi(P_{n}(\Id)(\Sigma S^V)) \}_n$ is a finite $K$-local Weiss tower for $\Phi(\Sigma S^V)$.
\end{prop}

\begin{proof}
The fibers of the tower $\{ \Phi(P_{n}(\Sigma S^V)) \}_n$ are given by
$$ \DD_{n}(\Id)(\Sigma S^V)_K \rightarrow \Phi(P_{n}(\Id)(\Sigma S^V)) \rightarrow \Phi(P_{n-1}(\Id)(\Sigma S^V)). $$
By \cite[Thm.~8.9]{Kuhn}, the map
$$ \Phi(\Sigma S^V) \rightarrow  \Phi(P_{p^h}(\Id)(\Sigma S^V)) $$
is an equivalence.
\end{proof}

\begin{prop}\label{prop:Weiss2}
The tower $\{F_{n} \TAQ_{S_K}(S_K^{(\Sigma S^V)_+}) \}$ obtained by taking the $K$-local Spanier-Whitehead dual of the Kuhn filtration $\{ F_n \TAQ^{S_K}(S_K^{(\Sigma S^V)_+})\}$ is a finite $K$-local Weiss tower. 
\end{prop}

\begin{proof}
By \ref{eq:Kuhncofiber}, the fibers of the tower are given by
\begin{align*}
F(\partial_n(\Id)^\vee \wedge_{h\Sigma_n} (S_K^{\Sigma S^V})^{\wedge_{S_K} n}, S_K)
& \simeq F(\partial_n(\Id)^\vee \wedge_{h\Sigma_n} (S^{\Sigma S^V})^{n}, S_K)
\\
&  \simeq F(\partial_n(\Id)^\vee \wedge (S^{\Sigma S^V})^{\wedge n}, S_K)^{h\Sigma_n} \\
&  \simeq (F(\partial_n(\Id)^\vee \wedge (S^{\Sigma S^V})^{\wedge n}, S_K)^{h\Sigma_n})_K \\
&  \simeq (((\partial_n(\Id) \wedge S^n \wedge S^{nV})_K)^{h\Sigma_n})_K \\
&  \simeq  ((\partial_n(\Id) \wedge S^n \wedge S^{nV})_{h\Sigma_n})_K.
\end{align*}
Thus they are equivalent to $K$-localizations of homogeneous degree $n$ functors.  Since we have
$$ \TAQ^{S_K}(S_K^{(\Sigma V)_+}) \simeq \hocolim_n F_n \TAQ^{S_K}(S_K^{(\Sigma V)_+}),  $$
we have
$$ \TAQ_{S_K}(S_K^{(\Sigma S^V)_+}) \simeq \holim_n F_{n} \TAQ_{S_K}(S_K^{(\Sigma S^V)_+}). $$
Since the layers are equivalent to $\DD_{n}(\Id)(\Sigma S^{V})_K$, they are acyclic for $n > p^h$  \cite{AroneMahowald}.
\end{proof}

\section{The comparison map on odd spheres}\label{sec:main}

Fix $q$ to be an odd positive integer.  The main result of this section is the following theorem.

\begin{thm}\label{thm:main}
The comparison map
$$ c^{S_K}: \Phi(S^q) \rightarrow \TAQ_{S_K}(S_K^{S^q_+}) $$
is an equivalence.
\end{thm}

We shall begin with its dual, and establish the following weaker statement.

\begin{lem}\label{lem:section}
For $\mr{dim} V \gg 0$, the natural transformation
$$ c_{S_K}: \TAQ^{S_K}(S_K^{(\Sigma S^V)_+}) \rightarrow S_K^{\Phi(\Sigma S^V)}$$
of functors from complex vector spaces to $K$-local spectra has a weak section: there is a natural transformation
$$ s:  S_K^{\Phi(\Sigma S^V)} \rightarrow \TAQ^{S_K}(S_K^{(\Sigma S^V)_+}) $$
so that $c_{S_K} \circ s$ is an equivalence.
\end{lem}

\begin{proof}
Using work of Arone-Mahowald, Kuhn shows that the map
$$ \Phi(X) \rightarrow \Phi(P_{p^h}(\Id)(X)) $$
is an equivalence \cite[Thm.~8.9]{Kuhn}.
Let $X = \Sigma S^V$, and let
$$ X \rightarrow Q^{\bullet+1} X $$
denote the Bousfield-Kan cosimplicial resolution.
Consider the diagram:
\begin{equation}\label{eq:crucialdiagram}
\xymatrix{
\TAQ^{S_K}(S_K^{X_+}) \ar[r]^-{c_{S_K}} &
S_K^{\Phi (X)}
\\ 
& S_K^{\Phi P_{p^h}(\Id)(X)} \ar@{=}[u]
\\
\TAQ^{S_K}(S_K^{P_{p^h}(Q^{\bullet+1})(X)_+}) \ar[uu] \ar[r]^-{c_{S_K}} &
S_K^{\Phi P_{p^h}(Q^{\bullet+1}) X} \ar[u]
}
\end{equation}
It is well known (see \cite{AroneKankaanrinta}) that there is an equivalence of cosimplicial $\Sigma_n$-spectra:
$$ \partial_n(Q^{\bullet+1}) \simeq \Sigma^\infty \mc{P}(n)^\vee_\bullet $$
so that the induced map
$$  \partial_n(\Id) \xrightarrow{\simeq} \Tot \partial_n(Q^{\bullet+1}) \simeq \Tot \Sigma^\infty \mc{P}(n)^\vee_\bullet \simeq  \Sigma^\infty \abs{\mc{P}(n)_\bullet}^\vee $$
is equivalence (\ref{eq:AroneMahowald}).
For a fixed $s$, the iterated Snaith splitting implies that the Goodwillie tower for $Q^{s+1}$ splits, giving equivalences
\begin{align*}
P_{p^h}(Q^{\bullet+1})(X) & \simeq \prod_{1 \le i \le p^h} Q(\mc{P}(i)_s \wedge_{h\Sigma_i} X^{\wedge i}) \\
& \simeq  Q \left( \bigvee_{1 \le i \le p^h} \mc{P}(i)_s \wedge_{h\Sigma_i} X^{\wedge i} \right).
\end{align*}
In particular, for $\mr{dim} V \gg 0$, the spaces above satisfy the hypotheses of Prop.~\ref{prop:comparisonQ}, and the comparison map
$$
\TAQ^{S_K}(S_K^{P_{p^h}(Q^{\bullet+1})(X)_+}) \xrightarrow{c_{S_K}}
S_K^{\Phi P_{p^h}(Q^{\bullet+1}) X}
$$
is a levelwise equivalence of simplicial spectra.
It follows from Diagram~(\ref{eq:crucialdiagram}) that the natural map 
$$ \abs{S_K^{\Phi P_{p^h}(Q^{\bullet+1})(X)}}_K \rightarrow S_K^{\Phi P_{p^h}(\Id)(X)} \simeq S_K^{\Phi(X)} $$ 
factors through $c_{S_K}$:
$$ \abs{S_K^{\Phi(P_{p^h}(Q^{\bullet+1})(X))}}_K \rightarrow \TAQ^{S_K}(S_K^{X_+}) \xrightarrow{c_{S_K}} S^{\Phi(X)}_K \simeq S^{\Phi(P_{p^h}(\Id)(X))}_K. $$
The lemma will be proven if we can show that the natural map
$$ \abs{S_K^{\Phi(P_{p^h}(Q^{\bullet+1})(X))}}_K \rightarrow  S^{\Phi(P_{p^h}(\Id)(X))}_K $$
is an equivalence.  To do this, we will prove that for all $n$ the map
$$ \abs{S_K^{\Phi(P_{n}(Q^{\bullet+1})(X))}}_K \rightarrow  S^{\Phi(P_{n}(\Id)(X))}_K $$
is an equivalence, by induction on $n$.  The map of fiber sequences
$$
\xymatrix{
D_n(\Id)(X) \ar[r] \ar[d] &
P_n(\Id)(X) \ar[r] \ar[d] &
P_{n-1}(\Id)(X) \ar[d]
\\
D_n(Q^{\bullet+1})(X) \ar[r] & 
P_n(Q^{\bullet+1})(X) \ar[r] &
P_{n-1}(Q^{\bullet+1})(X)
}
$$
gives a map of fiber sequences
$$
\xymatrix{
\abs{S_K^{\Phi(P_{n-1}(Q^{\bullet+1})(X))}}_K
 \ar[r] \ar[d] & 
\abs{S_K^{\Phi(P_n(Q^{\bullet+1})(X))}}_K \ar[r] \ar[d]&
\abs{S_K^{\Phi(D_n(Q^{\bullet+1})(X))}}_K \ar[d] 
\\
S_K^{\Phi(P_{n-1}(\Id)(X))} \ar[r] &
S_K^{\Phi(P_n(\Id)(X))} \ar[r] &
S_K^{\Phi(D_n(\Id)(X))}
}
$$
The induction on $n$ therefore rests on proving that the natural map
$$ \abs{S_K^{\DD_n(Q^{\bullet+1})(X)}}_K \simeq \abs{S_K^{\Phi(D_n(Q^{\bullet+1})(X))}}_K \rightarrow S_K^{\Phi(D_n(\Id)(X))} \simeq  S_K^{\DD_n(\Id)(X)} $$
is an equivalence.

Using the finiteness of $X$ and $\mc{P}(n)_\bullet$, together with the vanishing of $K$-local Tate spectra \cite{KuhnTate}, we have the following diagram of equivalences
$$
\xymatrix{
\abs{S_K^{\DD_n(Q^{\bullet+1})(X)}}_K \ar[r] &
S_K^{\DD_n(\Id)(X)} 
\\
\abs{S_K^{\Sigma^\infty \mc{P}(n)_\bullet^\vee \wedge_{h\Sigma_n} X^{\wedge n}}}_K\ar[r] \ar[u]^\simeq &
S_K^{\Sigma^\infty \abs{\mc{P}(n)_\bullet}^\vee \wedge_{h\Sigma_n} X^{\wedge n}} \ar[u]^\simeq
\\
\abs{\left( S_K^{\Sigma^\infty \mc{P}(n)_\bullet^\vee \wedge X^{\wedge n}}\right)^{h\Sigma_n} }_K\ar[r] \ar[u]^\simeq &
\left( S_K^{\Sigma^\infty \abs{\mc{P}(n)_\bullet}^\vee \wedge X^{\wedge n}} \right)^{h\Sigma_n} \ar[u]^\simeq
\\
\abs{\left( S_K^{\Sigma^\infty \mc{P}(n)_\bullet^\vee \wedge X^{\wedge n}}\right)_{h\Sigma_n} }_K\ar[r] \ar[u]^\simeq &
\left\lbrack \left( S_K^{\Sigma^\infty \abs{\mc{P}(n)_\bullet}^\vee \wedge X^{\wedge n}} \right)_{h\Sigma_n} \right\rbrack_K \ar[u]^\simeq
\\
\abs{\left( \Sigma^\infty \mc{P}(n)_\bullet \wedge S_K^{ X^{\wedge n}}\right)_{h\Sigma_n} }_K\ar[r] \ar[u]^\simeq &
\left\lbrack \left( \Sigma^\infty \abs{\mc{P}(n)_\bullet} \wedge S_K^{ X^{\wedge n}} \right)_{h\Sigma_n} \right\rbrack_K \ar[u]^\simeq
}
$$
The bottom arrow in this diagram is an equivalence, since realizations commute past homotopy colimits and smash products.  Therefore the top arrow in the diagram is an equivalence, as desired.
\end{proof}

The final ingredient we will need to prove Theorem~\ref{thm:main} will be a result which will allow us to dualize Lemma~\ref{lem:section}.

\begin{prop}\label{eq:dualizable}
The spectrum $\Phi(S^q)$ is $K$-locally dualizable.
\end{prop}

\begin{proof}
It suffices to show that its completed Morava $E$-homology is finitely generated \cite{HoveyStrickland}.  Since $\Phi (S^q) \simeq \Phi (P_{p^h}(\Id)(S^q))$ \cite[Sec.~7]{Kuhn}, one can prove this by proving $\Phi(P_{p^k}(\Id)(S^q))$ has finitely generated completed Morava $E$-homology  by induction on $k$.  This is done using the fiber sequences
$$ 
\DD_{p^k}(\Id)(S^q)_K \rightarrow \Phi(P_{p^k}(\Id)(S^q)) \rightarrow \Phi(P_{p^{k-1}}(\Id)(S^q))
$$
together with our computation $E_0 L(k)_q \cong C[k]_{-q}^\vee$.  Note that $C[k]^\vee_{-q}$ is finitely generated by \cite[Prop.~4.6]{RezkKoszul}. 
\end{proof}

\begin{proof}[Proof of Theorem~\ref{thm:main}]
For $\mr{dim} V \gg 0$, we can take the $K$-local Spanier-Whitehead dual of the retraction
$$ S_K^{\Phi(\Sigma S^V)} \rightarrow \TAQ^{S_K}(S_K^{(\Sigma S^V)_+}) \xrightarrow{c_{S_K}} S_K^{\Phi(\Sigma S^V)} $$ 
provided by Lemma~\ref{lem:section} to obtain a retraction of functors from complex vector spaces to $K$-local spectra:
$$ 
\Phi(\Sigma S^V) \xrightarrow{c^{S_K}} \TAQ_{S_K}(S_K^{(\Sigma S^V)_+}) \rightarrow \Phi(\Sigma S^V).
$$
We therefore get a retraction of the $K$-local Weiss towers of these functors, restricted to $\mr{dim} V \gg 0$ (see Propositions~\ref{prop:Weiss1} and \ref{prop:Weiss2})
$$ 
\{ \Phi(P_{n}(\Id)(\Sigma S^V)) \}_n \xrightarrow{c^{S_K}} 
\{F_{n} \TAQ_{S_K}(S_K^{(\Sigma S^V)_+}) \}_n \rightarrow 
\{ \Phi(P_{n}(\Id)(\Sigma S^V)) \}_n.
$$
However, the layers of both of these towers are equivalent to the spectra $\DD_{n}(\Id)(\Sigma S^V)_K$.  Since the Morava $K$-theory of these layers is finite, it follows that the map $c^{S_K}$ induces an equivalence on the layers of the $K$-local Weiss towers for $\mr{dim} V \gg 0$.  It follows that the comparison map induces an equivalence on Weiss derivatives, from which it follows that the comparison map is an equivalence on layers for \emph{all} $V$.
Since the $K$-local Weiss towers are themselves finite, we deduce that the natural transformation
$$
\Phi(\Sigma S^V) \xrightarrow{c^{S_K}} \TAQ_{S_K}(S_K^{(\Sigma S^V)_+})
$$
is an equivalence by inducting up the towers.
\end{proof}

Actually, the method of proof gives the following corollary, which allows us to compare $\Phi$ applied to the Goodwillie tower of the identity with the much easier to understand Kuhn tower.

\begin{cor}\label{cor:towercomp}
The comparison map induces an equivalence of towers
$$ 
\{ \Phi (P_n(\Id)(S^q)) \} \xrightarrow[\simeq]{c^{S_K}} \{F_{n} \TAQ_{S_K}(S_K^{S^q_+}) \}.
$$
\end{cor}

\section{The Morava $E$-homology of the Goodwillie attaching maps}\label{sec:kinv}

Fix $q$ to be an odd positive integer.
Let $\alpha_k$ denote the attaching map connecting the $p^k$ and $p^{k+1}$-layers of the Goodwillie tower for $S^q$.
$$
\xymatrix@R-1em{
\alpha_k: D_{p^k}(\Id)(S^q) \ar[r] \ar@{=}[d]
& BD_{p^{k+1}}(\Id)(S^q) \ar@{=}[d] \\
 \Omega^\infty \Sigma^{q-k} L(k)_q & 
\Omega^\infty \Sigma^{q-k} L(k)_q
}
$$
Applying $\Phi$ and desuspending, we get a map
$$ \Phi(\alpha_k): (L(k)_q)_K \rightarrow (L(k+1)_q)_K $$
which should be regarded as the corresponding attaching map between consecutive non-trivial layers in the $v_h$-periodic Goodwillie tower of the identity.

Note that since $E^{S^q_+}$ is a commutative $E$-algebra, the reduced cohomology group
$$ \td{E}^q(S^q) = V^q \pi_* E^{S^q_+} $$
is a $\Delta^q$-module.
Under the isomorphisms
$$ E_0 L(k)_q \cong \left( C[k]_{-q} \otimes_{E_0} \td{E}^q(S^q) \right)^\vee $$
obtained by tensoring the isomorphism of Theorem~\ref{thm:enlk} with the fundamental class in $\td{E}^q(S^q)$, there is an induced map
$$ E_0 \Phi(\alpha_k): \left( C[k]_{-q} \otimes_{E_0} \td{E}^q(S^q) \right)^\vee \rightarrow \left( C[k+1]_{-q} \otimes_{E_0} \td{E}^q(S^q) \right)^\vee. $$
We have the following more refined version of Theorem~\ref{thm:enlk}.

\begin{thm}\label{thm:koszul}
There is an isomorphism of cochain complexes
$$  (E_0 L(k)_q, E_0 \Phi(\alpha_k)) \cong (C_k^{\Delta^q}(\td{E}^q(S^q))^\vee, \delta^\vee_k) $$
where $C_k^{\Delta^q}(\td{E}^q(S^q))$ is the Koszul complex for the $\Delta^q$-module $\td{E}^q(S^q)$.
\end{thm}

\begin{proof}
By Corollary~\ref{cor:towercomp}, it suffices to show that the $E$-homology of the attaching maps in the Kuhn tower
$$ \alpha'_k: (L(k)_q)_K \simeq \left( \frac{F_{p^k} \TAQ^{S_K}(S^{S^q_+}_K)}{F_{p^{k-1}} \TAQ^{S_K}(S^{S^q_+}_K)} \right)^\vee \rightarrow \left( \frac{F_{p^{k+1}} \TAQ^{S_K}(S^{S^q_+}_K)}{F_{p^{k}} \TAQ^{S_K}(S^{S^q_+}_K)} \right)^\vee \simeq (L(k+1)_q)_K  
 $$
has the desired description (here the $(-)^\vee$ notation above denotes the \emph{$K$-local} Spanier-Whitehead dual).  The result is obtained by dualizing the following diagram
$$
\xymatrix{
C[k+1]_{-q} \otimes_{E_0} \td{E}^q(S^q) \ar[r]^{\delta_k} \ar@{^{(}->}[d] &
C[k]_{-q} \otimes_{E_0} \td{E}^q(S^q) \ar@{^{(}->}[d]
\\
B_{k+1}(\bar{E}_0, \Delta^q, \td{E}^q(S^q))[k+1] \ar[r]^{d_{k+1}} 
& B_{k}(\bar{E}_0, \Delta^q, \td{E}^q(S^q))[k]
\\ 
B_{k+1}(\Id, \bar{\TT}, \td{E}^*(S^q)) \bra{p^{k+1}}_{-q} \ar[r]^{d_{k+1}\bra{p^k}} \ar[u]^{\mc{L}}
& B_{k}(\Id, \bar{\TT}, \td{E}^q(S^q))\bra{p^{k}}_{-q} \ar[u]_{\mc{L}} 
\\
E_{-q} B_{k+1}(\Id, \td{\PP}, I(S^{S^q_+})) \bra{p^{k+1}} \ar[r]^-{d_{k+1}\bra{p^k}} \ar@{=}[u]
& E_{-q} B_{k}(\Id, \td{\PP}, I(S^{S^q_+})) \bra{p^{k}} \ar@{=}[u]
}
$$
which identifies the $E$-homology of the attaching map
$$
\frac{F_{p^{k+1}} \TAQ^{S_K}(S^{S^q_+}_K)}{F_{p^{k}} \TAQ^{S_K}(S^{S^q_+}_K)}
\rightarrow
 \frac{F_{p^k} \TAQ^{S_K}(S^{S^q_+}_K)}{F_{p^{k-1}} \TAQ^{S_K}(S^{S^q_+}_K)}. 
$$
In this diagram, the maps $d_{k+1}$ are the last face maps in the corresponding bar complexes, and the maps $d_{k+1}\bra{p^k}$ are the projections of the face maps on to the $\bra{p^k}$-summands.
\end{proof}

\begin{cor}\label{cor:koszul}
The spectral sequence obtained by applying $E_*$ to the tower $\{ \Phi (P_n(\Id)(S^q)) \}$ takes the form
$$ \Ext^{s}_{\Delta^q}(\td{E}^q(S^q), \bar{E}_t)  \Rightarrow E_{q+t-s} \Phi(S^q). $$
\end{cor}

\section{A modular description of the Koszul complex}\label{sec:modular}

\subsection*{Reduction to the case of $q = 1$}

In this section we give a modular interpretation of the Koszul complex $C_*^{\Delta^q}(\td{E}^q(S^q))$ in the case of $q = 1$.  Since the suspension gives inclusions of bar complexes (see (\ref{eq:zigzag}))
$$ B(\bar{E}_0, \Delta^q, \td{E}^q(S^q)) \hookrightarrow B(\bar{E}_0, \Delta^1, \td{E}^1(S^1)) $$
we deduce that we have an induced map of Koszul complexes
$$
\xymatrix{
C[k]_{-q} \ar@{^{(}->}[r] \ar@{.>}[d]_{\sigma^{q-1}} & 
\Delta^q[1]^{\otimes k} \ar@{^{(}->}[d] 
\\
C[k]_{-1} \ar@{^{(}->}[r] & 
\Delta^1[1]^{\otimes k}
}
$$
Furthermore, the map $\sigma^{q-1}$ above must be an inclusion.  We deduce that there is an inclusion of Koszul complexes
$$ \sigma^{q-1}: C_*^{\Delta^q}(\td{E}^q(S^q)) \hookrightarrow C_*^{\Delta^1}(\td{E}^1(S^1)). $$
It follows that the modular description of the Koszul complex we shall give for $q = 1$ will extend to a modular description for arbitrary odd $q$ provided we have a good understanding of the inclusions of lattices
\begin{align*}
\Delta^q[1] & \subseteq \Delta^1[1], \\
\Delta^q[2] & \subseteq \Delta^1[2].
\end{align*}
This amounts to having a concrete understanding of the second author's ``Wilkerson Criterion'' \cite{RezkWilk}.

\subsection*{The modular isogeny complex}

We review the definition of the modular isogeny complex $\mc{K}^*_{p^k}$ of \cite{RezkMIC} associated to the formal group $\GG$.  

For $(k_1, \ldots, k_s)$ a sequence of positive integers, let
$$ \mr{Sub}_{p^{k_1}, \ldots, p^{k_s}}(\GG) = \Spf(\mc{S}_{p^{k_1}, \ldots , p^{k_s}}) $$
be the (affine) formal scheme whose $R$-points are given by
$$ \mr{Sub}_{p^{k_1}, \ldots, p^{k_s}}(\GG)(R) = \{ H_1 < \cdots < H_s < \GG \times_{\Spf(E_0)} \Spf(R) \: : \: \abs{H_i/H_{i-1}} = p^{k_i} \}. $$

\begin{lem}\label{lem:decomp}
There is a canonical isomorphism of $E_0$-algebras
$$  \mc{S}_{p^{k_1}, \ldots, p^{k_s}} \cong \mc{S}_{p^{k_1}} \otimes_{E_0} \cdots \otimes_{E_0} \mc{S}_{p^{k_s}}.  $$
\end{lem}

\begin{proof}
An $R$ point of $\Spf(\mc{S}_{p^{k_1}, \ldots, p^{k_s}})$ corresponds to a chain of finite subgroups
$$ (H_1 < \cdots < H_s) $$
in $\GG_1 := \GG \times_{\Spf(E_0)} \Spf(R)$ with $\abs{H_i/H_{i-1}} = p^i$.  Define 
$ \GG_i := \GG_1/H_{i-1} $.  Then, defining, $\td{H}_i := H_i/H_{i-1}$, we get a collection of $R$-points  
$$ (\GG_i, \td{H}_i) \in \Spf(\mc{S}_{p^i})(R) $$
and isomorphisms $\GG_i/\td{H}_i \cong \GG_{i+1}$.  This is precisely the data of an $R$-point of 
$$ \Spf(\mc{S}_{p^{k_1}} \otimes_{E_0} \cdots \otimes_{E_0} \mc{S}_{p^{k_s}}). $$
Conversely, given such a sequence $(\GG_i, \td{H}_i)$ with isomorphisms $\GG_i/\td{H}_i \cong \GG_{i+1}$, there is an associated chain of subgroups $(H_1, \ldots, H_s)$ of $\GG_1$ obtained by pulling back the subgroup $\td{H}_i$ over the isogeny:
$$ \GG_1 \rightarrow \GG_1/\td{H}_1 \cong \GG_2 \rightarrow \GG_2/\td{H}_2 \cong \GG_3 \rightarrow \cdots \rightarrow \GG_{i-1}/\td{H}_{i-1} \cong \GG_i. $$
\end{proof}

For $k > 0$ we define
$$ \mc{K}^s_{p^k} = 
\begin{cases}
\prod_{\substack{k_1 + \cdots + k_s = k \\ k_i  > 0}} \mc{S}_{p^{k_1}, \ldots , p^{k_s}}, & 1 \le s \le k, \\
0, & \text{otherwise}.
\end{cases}
$$
We handle the case of $k = 0$ by defining 
$$
\mc{K}^s_{1} = 
\begin{cases}
E_0, & s = 0, \\
0, & s > 0.
\end{cases}
$$
For $1 \le i \le s$ and a decomposition $k_i = k_i'+k_i''$  with $k_i', k_i'' > 0$, define maps
$$ u_i: \mc{S}_{p^{k_1}, \ldots, p^{k_s}} \rightarrow \mc{S}_{p^{k_1}, \ldots, p^{k'_i}, p^{k''_i}, \ldots, p^{k_s}} $$
on $R$ points by
$$ u_i^*: (H_1 < \cdots < H_{s+1}) \mapsto (H_1 < \cdots < \widehat{H}_i < \cdots < H_{s+1}). $$
The maps $u_i$, under the isomorphism of Lemma~\ref{lem:decomp}, all arise from the maps
$$ u_1: \mc{S}_{p^{k'+k''}} \rightarrow \mc{S}_{p^{k'}} \otimes_{E_0} \mc{S}_{p^{k''}}. $$
In \cite{RezkWilk}, it is established that the maps $u_1$ above are dual to the algebra maps
$$ \Gamma[k'] \otimes_{E_0} \Gamma[k''] \rightarrow \Gamma[k'+k'']. $$

Taking a product over all possible such decompositions of $k_i = k_i' + k_i''$ gives a map
$$ u_i: \mc{K}^s_{p^k} \rightarrow \mc{K}^{s+1}_{p^k}. $$
The differentials 
$$ \delta: \mc{K}^s_{p^k} \rightarrow \mc{K}^{s+1}_{p^k}, \quad 1 \le s < k $$
in the cochain complex $\mc{K}^*_{p^k}$ are given by
$$ \delta(x) = \sum_{1 \le i \le s} (-1)^i u_i(x). $$

\subsection*{The cohomology of the modular isogeny complex}

The key observation of this section is the following.

\begin{prop}\label{prop:HMIC}
There is an isomorphism of cochain complexes
$$ B_*(\bar{E}_0, \td{\Delta}^1, \bar{E}_0)[k]^\vee \cong \mc{K}^*_{p^k}. $$
It follows that we have
$$
H^s(\mc{K}^*_{p^k}) \cong
\begin{cases}
C[k]_{-1}^\vee, & s = k, \\
0, &  s \ne k.
\end{cases}
$$
\end{prop}

\begin{proof}
The suspension isomorphism (\ref{eq:susp})
$$ \sigma: \Delta^1 \xrightarrow{\cong} \Gamma^0 $$
induces an isomorphism of chain complexes
$$ B_*(\bar{E}_0, \td{\Delta}^1, \bar{E}_0)[k] \cong B_*(\bar{E}_0, \td{\Gamma}^0, \bar{E}_0)[k]. $$
The isomorphisms (\ref{eq:dualGamma}) together with those of Lemma~\ref{lem:decomp}
induce isomorphisms
\begin{align*}
B_s(\bar{E}_0, \td{\Gamma}^0, \bar{E}_0)[k]^\vee & = 
\left( \bigoplus_{\substack{k_1 + \cdots + k_s = k \\ k_i > 0}} \td{\Gamma}^0[k_1] \otimes_{E_0} \cdots \otimes_{E_0} \td{\Gamma}^0[k_s] \right)^\vee \\
& \cong \bigoplus_{\substack{k_1 + \cdots + k_s = k \\ k_i > 0}} \td{\Gamma}^0[k_1]^\vee \otimes_{E_0} \cdots \otimes_{E_0} \td{\Gamma}^0[k_s]^\vee \\
& \cong \bigoplus_{\substack{k_1 + \cdots + k_s = k \\ k_i > 0}} \mc{S}_{p^{k_1}} \otimes_{E_0} \cdots \otimes_{E_0} \mc{S}_{p^{k_s}} \\
& \cong \prod_{\substack{k_1 + \cdots + k_s = k \\ k_i > 0}} \mc{S}_{p^{k_1}, \ldots, p^{k_s}}
\end{align*}
since all of the $E_0$-modules involved are finite and free.  Using the facts that $\Gamma^0[t]$ acts trivially on $\bar{E}_0$ for $t >0$, and that the differential in the modular isogeny complex is an alternating sum of maps dual to the multiplication maps in $\Gamma^0$, our isomorphisms yield the desired isomorphism of cochain complexes
$$ B_*(\bar{E}_0, \td{\Delta}^1, \bar{E}_0)[k]^\vee \cong \mc{K}^*_{p^k}. $$
Again, appealing the the fact that these cochain complexes are free $E_0$-modules in each degree, and that the modules $C[k]_{-1}$ are free (see \cite[Prop.~4.6]{RezkKoszul}), we have
\begin{align*}
H^s(\mc{K}^*_{p^k}) & \cong  H^s(B_*(\bar{E}_0, \td{\Delta}^1, \bar{E}_0)[k]^\vee) \\
& \cong H_s(B_*(\bar{E}_0, \td{\Delta}^1, \bar{E}_0)[k])^\vee \\
& 
\cong 
\begin{cases}
C[k]_{-1}^\vee, & s = k, \\
0, & s \ne k.
\end{cases}
\end{align*}
\end{proof}

\subsection*{Modular description of the Koszul differentials}

What remains is to give a modular description of the Koszul differentials
$$ H^k(\mc{K}^*_{p^k}) \cong C^{\Delta^1}_k(\td{E}^1(S^1))^\vee \xrightarrow{\delta_k^\vee} C^{\Delta^1}_{k+1}(\td{E}^1(S^1))^\vee \cong H^{k+1}(\mc{K}^*_{p^{k+1}}). $$
Consider the map
$$ u_{k+1} : \mc{S}_{\underbrace{p, \ldots, p}_k} \rightarrow \mc{S}_{\underbrace{p, \ldots, p}_{k+1}} $$
whose effect on $R$-points is given by
$$ u^*_{k+1}:(H_1 < \cdots < H_{k+1} < \GG) \mapsto (H_2/H_1 < \cdots < H_{k+1}/H_1 < \GG/H_1). $$

\begin{thm}\label{thm:modularkoszuldiffs}
The following diagram commutes.
$$
\xymatrix{
\mc{S}_{\underbrace{p, \ldots, p}_k} \ar[r]^{u_{k+1}} \ar@{->>}[d]
&  \mc{S}_{\underbrace{p, \ldots, p}_{k+1}} \ar@{->>}[d]
\\
C_k^{\Delta^1}(\td{E}^1(S^1))^\vee \ar[r]_{\delta_k^\vee} &
C_{k+1}^{\Delta^1}(\td{E}^1(S^1))^\vee
}
$$
\end{thm}

\begin{proof}
Under the suspension isomorphism $\sigma: \Delta^1 \cong \Gamma^0$, the $\Delta^1$-module $\td{E}^1(S^1)$ is isomorphic to the $\Gamma^0$-module $\td{E}^0(S^0) = E_0$.  Moreover the action map
$$ \Gamma^0[1] \cong \Gamma^0[1] \otimes_{E_0} E_0 \rightarrow E_0 $$
is dual to the map $t$ of (\ref{eq:t})
$$ t: E_0 \rightarrow \mc{S}_{p} $$
whose effect on $R$ points is given by
$$ t^*: (H < \GG) \mapsto \GG/H. $$
The result follows from the isomorphisms
\begin{align*}
\mc{S}_{\underbrace{p, \ldots, p}_k} & \cong \underbrace{\mc{S}_p \otimes_{E_0} \cdots  \otimes_{E_0} \mc{S}_p}_k \\
& \cong  B_k(\bar{E}_0, \td{\Gamma}^0, E_0)[k] \\
& \cong B_k(\bar{E}_0, \td{\Delta}^0, \td{E}^1(S^1))[k]
\end{align*}
and (\ref{eq:Cdiff}).
\end{proof}

\appendix
\section{Borel equivariant stable homotopy theory}\label{apx:borel}

The technical aspects of Appendix~\ref{apx:norm} will necessitate a detailed understanding of homotopy orbit and fixed point spectra, including norm and transfer maps in the stable homotopy category, and a theory of Euler classes.  Everything in this appendix should be well known, but it seems difficult to fully track down in the literature.  The first author learned of this particular perspective on norms from some lectures of Jacob Lurie.

\subsection*{Functors induced from homomorphisms}

For a finite group $G$, let $\Sp_G$ denote the category of $G$-spectra ($G$-equivariant objects in $\Sp$, with weak equivalences given by those equivariant maps which are equivalences on underlying non-equivariant spectra), and $\mr{Ho}(\Sp_G)$ the corresponding homotopy category.

Given a homomorphism $f : H \rightarrow G$, the associated restriction functor
$$ f^* : \mr{Ho}(\Sp_G) \rightarrow \mr{Ho}(\Sp_H) $$
has a left adjoint 
$$ f_! : \mr{Ho}(\Sp_H) \rightarrow \mr{Ho}(\Sp_G) $$
and a right adjoint 
$$ f_* : \mr{Ho}(\Sp_H) \rightarrow \mr{Ho}(\Sp_G). $$
In the case where $f: H \rightarrow G$ is the inclusion of a subgroup, these functors are given by induction and coinduction
\begin{align*}
f_! Y & = \mr{Ind}_H^G Y = G_+ \wedge_H Y, \\
f_* Y & = \mr{CoInd}_H^G Y = \Map_H(G, Y). \\
\end{align*}
In this special case, since finite products are equivalent to finite wedges in $\Sp$, the natural map
$$ f_! Y = \mr{Ind}_H^G Y \xrightarrow[\cong]{\psi_f} \mr{CoInd}_H^G Y = f_* Y $$
is an isomorphism in $\mr{Ho}(\Sp_G)$, and thus $f_!$ is also right adjoint to $f^*$.

If $f$ is the unique map to the trivial group $f: G \rightarrow 1$, then these functors are given by homotopy orbits and homotopy fixed points:
\begin{align*}
f_! Y & = Y_{hG}, \\
f_* Y & = Y^{hG}. \\
\end{align*}

In general, these functors are compatible with composition:
\begin{align*}
(fg)^* &= g^* f^*, \\
(fg)_! & = f_! g_!, \\
(fg)_* & = f_* g_*. 
\end{align*}
For $Y_1$ and $Y_2$ in $\Sp_G$, let $Y_1 \wedge Y_2 \in \Sp_G$ denote the smash product with diagonal $G$-action.  For $f : H \rightarrow G$, $Y \in \Sp_H$, and $Z \in \Sp_G$, there is a projection formula
$$ Y \wedge (f_! Z) \cong f_!((f^* Y) \wedge Z). $$

Finally, if 
$$
\xymatrix{
H \ar[r]^f \ar[d]_g & G \ar[d]^{g'} \\
H' \ar[r]_{f'} & G'
}
$$
is a pullback, and $g'$ is surjective, then for $Y \in \Sp_G$ there is an isomorphism
$$ g_! f^* Y \cong (f')^* (g')_! Y. $$

For example, if $f: G \rightarrow G/N$ is a quotient, then for $Y \in \Sp_G$, $f_!Y$ is a $G/N$-equivariant model for $Y_{hN}$.  Indeed, this can be seen formally by considering the following diagram.
$$
\xymatrix{
N \ar[r]^k \ar[d]_i & 1 \ar[d]^j \\
G \ar[r]^f \ar[dr]_h & G/N \ar[d]^g \\
& 1 
}
$$
Since the square in the above diagram is a pullback, we deduce
$$ j^* f_! Y \cong k_! i^* Y = Y_{hN}. $$
Furthermore, we get an iterated homotopy orbit theorem
$$ (Y_{hN})_{hG/N} = g_! f_! Y = h_! Y = Y_{hG}. $$

\subsection*{Norm and transfer maps}

In the language introduced in the previous subsection, norm and transfer maps have a particularly nice description.
Suppose that $H$ is a subgroup of $G$, and consider the diagram:
$$
\xymatrix{
H \ar[rr]^i \ar[dr]_f && G \ar[dl]^g \\
& 1
}
$$
For $Y \in \Sp_G$, the transfer is given by the composite
$$ \mr{Tr}_H^G : Y_{hG} = g_! Y \xrightarrow{\eta_H^G} g_! i_! i^* Y \cong f_! i^* Y = Y_{hH}. $$
where $\eta_H^G$ is the composite
\begin{equation}\label{eq:eta}
\eta_H^G: Y \rightarrow i_* i^* Y \xrightarrow[\cong]{\psi^{-1}_{i}} i_! i^* Y
\end{equation}
arising from the unit of the adjunction.

If $H$ is normal in $G$, there is a refinement of the transfer $\Tr_e^H$ which is $G$-equivariant.  Consider the diagram
$$
\xymatrix{
G \ar[dr]|\Delta \ar[drr]^{\mr{Id}} \ar[ddr]_{\mr{Id}} \\
& G \times_{G/H} G \ar[r]_-{\pi_2} \ar[d]^{\pi_1} & G \ar[d]^f \\
& G \ar[r]_f & G/H
}
$$
Using the fact that the square in the diagram is a pullback, we define the $G$-equivariant transfer to be the composite
\begin{multline}\label{eq:equivtr}
\Tr_e^H : f^* Y_{hH} = f^* f_! Y = (\pi_1)_! (\pi_2)^* Y \xrightarrow{\eta^{G \times_{G/H} G}_{G}} (\pi_1)_! \Delta_! \Delta^* (\pi_2)^* Y = Y.
\end{multline}
The adjoint of this map gives a $G/H$-equivariant norm map
$$ N_H : Y_{hH} \rightarrow f_* Y = Y^{hH}. $$

The equivariant transfer maps (\ref{eq:equivtr}) can be constructed more generally: for subgroups
$$ K \le H \le G $$
with $K$ and $H$ normal in $G$
we can construct the $G/K$ equivariant transfer $\mr{Tr}^H_K$ as the composite
$$ \mr{Tr}_K^H : Y_{hH} \simeq (Y_{hK})_{hH/K} \xrightarrow{\mr{Tr}^{H/K}_e}  Y_{hK}. $$

We end this section with a lemma which we will need to make use of later.

\begin{lem}\label{lem:twotransfer}
Given $X, Y \in \Sp_G$, the following diagram commutes in $\mr{Ho}(\Sp_G)$.
$$
\xymatrix{
X_{hG} \wedge Y_{hG} \ar[r]^{\cong} \ar[d]_{\cong} & (X \wedge Y_{hG})_{hG} \ar[d]^{1 \wedge \mr{Tr}^G_e}
\\
 (X_{hG} \wedge Y)_{hG} \ar[r]_{\mr{Tr}_e^G \wedge 1} & (X \wedge Y)_{hG}
}
$$
\end{lem}

\begin{proof}
With respect to the maps:
$$
\xymatrix{
G \ar[dr]|{\delta} \ar@/_1pc/[ddr]_{=} \ar@/^1pc/[drr]^= & & \\
& G \times G \ar[d]_{\pi_1} \ar[r]^{\pi_2} \ar[dr]|g & G \ar[d]^f
\\
& G \ar[r]_f & 1 
}
$$
the lemma follows from the following commutative diagram.
$$
\scriptsize
\xymatrix@C-2em@R-1em{
f_!X \wedge f_!Y \ar[dd]_{\cong} \ar[rr]^\cong \ar[dr]_\cong && 
f_!(X \wedge f^*f_!Y) \ar[r]^\cong &
f_!(X \wedge (\pi_1)_! \pi^*_2 Y ) \ar@/^5pc/[dd] & 
\\
 &
g_!(\pi_1^* X \wedge \pi_2^* Y) \ar[d]_\cong \ar[dr] \ar[r]^\cong &
f_! (\pi_1)_! (\pi_1^*X \wedge \pi_2^*Y) \ar[ur]_{\cong} &
f_!(\pi_1)_! \delta_! (\delta^* \pi_1^* X \wedge \delta^* \pi_2^* Y) \ar[d]^\cong
\\
f_!(f^* f_! X \wedge Y) \ar[d]_\cong & 
f_! (\pi_2)_!(\pi_1^* X \wedge \pi_2^* Y) \ar[dl]_{\cong} &
g_! \delta_! \delta^* (\pi_1^* X \wedge \pi_2^* Y) \ar[dl]_\cong \ar[ru]^\cong &
f_!(X \wedge (\pi_1)_! \delta_! \delta^* \pi_2^* Y) \ar[d]^\cong &
\\
f_!((\pi_2)_! \pi_1^* X \wedge Y) \ar@/_2pc/[rr]
&f_! (\pi_2)_! \delta_! (\delta^* \pi_1^* X \wedge \delta^* \pi_2^* Y) \ar[r]_{\cong} &
f_!((\pi_2)_! \delta_! \delta^* \pi_1^* X \wedge Y) \ar[r]_\cong & f_!(X \wedge Y)
\\
 &&
 &&
}
$$
\end{proof}

\subsection*{Thom isomorphism and Euler classes}

For the purposes of this subsection, let $E$ be a complex orientable ring spectrum, with a fixed choice of complex orientation.  We may regard $E$ as a $G$-spectrum with trivial action.

For a $d$-dimensional complex representation $V$, its Thom class can be represented by a map
$$ [V]: S^V_{hG} \rightarrow \Sigma^{2d} E $$
in $\mr{Ho}(\Sp)$\footnote{Alternatively, the reader may assume that $E$ is even periodic, and that the Thom class lies in dimension $0$ --- the reader then may just set $d = 0$ in the formulas that follow.}.  By adjointness, this map corresponds to an equivariant map
$$ \td{[V]}: S^V \rightarrow \Sigma^{2d} E. $$
in $\mr{Ho}(\Sp_G)$.
By the definition of a Thom class, the induced map of $E$-modules
$$ \td{[V]}_E: E \wedge S^V \rightarrow \Sigma^{2d} E $$
is an underlying equivalence of spectra.  It follows that for any $Y \in \Sp_G$, there is an equivalence
$$ \td{[V]}_E: E \wedge S^V \wedge Y \rightarrow E \wedge \Sigma^{2d} Y. $$
It follows that there is an equivalence of non-equivariant spectra
$$ \Phi_V: E \wedge (S^V \wedge Y)_{hG} \rightarrow E \wedge (\Sigma^{2d} Y)_{hG}, $$
and thus isomorphisms
\begin{align}
(\Phi_V)_* : E_{*+2d} (S^V \wedge Y)_{hG} & \xrightarrow{\cong} E_* Y_{hG}, \\
\Phi_V^*: E^{*+2d} (S^V \wedge Y)_{hG} & \xleftarrow{\cong} E^* Y_{hG}. \label{eq:thom}
\end{align}
These isomorphisms are instances of the classical Thom isomorphism if $Y$ is of the form $\Sigma^{\infty} X_+$ for a $G$-space $X$.

Observe that, in general, $E^*Y_{hG}$ is a module over the ring $E^*(BG)$.  Indeed, given classes
\begin{align*}
\alpha & \in E^n(BG), \\
\beta & \in E^m Y_{hG}
\end{align*}
represented by stable maps
\begin{align*}
\alpha: S^0_{hG} & \rightarrow \Sigma^n E, \\
\beta : Y_{hG} & \rightarrow \Sigma^m E,
\end{align*}
we can take the smash of their adjoints
\begin{align*}
\td{\alpha}: S^0 & \rightarrow \Sigma^n E, \\
\td{\beta} : Y & \rightarrow \Sigma^m E 
\end{align*}
to get a map
$$ \td{\alpha} \wedge \td{\beta} : Y \rightarrow \Sigma^{n+m} E \wedge E. $$
Postcomposing with the product for $E$, and taking the adjoint, gives a map
$$ \alpha \cdot \beta : Y_{hG} \rightarrow \Sigma^{n+m} E $$
which represents the desired product
$$ \alpha \cdot \beta \in E^{n+m} Y_{hG}. $$

The composite
$$ S^0_{hG} \rightarrow S^V_{hG} \xrightarrow{[V]} \Sigma^{2d} E $$
represents the \emph{Euler class}
$$ e_V \in E^{2d}(BG). $$

\begin{lem}\label{lem:euler}
The composite
$$  E^* Y_{hG} \xrightarrow[\cong]{\Phi_V^*} E^{*+2d}(S^V \wedge Y)_{hG} \rightarrow E^{*+2d} Y_{hG} $$
induced by the inclusion $Y_{hG} \hookrightarrow (S^V \wedge Y)_{hG}$
is given by multiplication by $e_V$.
\end{lem}

\section{The $H_\infty$ structure of $S_K^{QX_+}$}\label{apx:norm}

To state the main result of this appendix, we shall need the following.

\begin{defn}We shall say that a pair of $\TT$-algebras $A$, $B$ are \emph{isomorphic mod $\mf{m}$} (and write $A \cong_\mf{m} B$) if there is a map of $E_*$-modules
$$ f : A \rightarrow B $$
such that 
\begin{enumerate}
\item the map
$$ \bar{f} : A/\mf{m} \rightarrow B/\mf{m} $$
is an isomorphism, and 
\item the following diagram commutes
$$
\xymatrix{
(\TT A)/\mf{m} \ar[r]^{\br{\TT f}} \ar[d] & (\TT B)/\mf{m} \ar[d] \\
A/\mf{m} \ar[r]_{\bar{f}} & B/\mf{m} 
}
$$
where the vertical maps are the mod $\mf{m}$ reductions of the $\TT$-algebra structure maps.
\end{enumerate}
\end{defn}

In this appendix we prove the following technical lemma needed in the proof of Proposition~\ref{prop:comparisonQ}.

\begin{lem}\label{lem:keylemma}
There is an $N \gg 0$ such that for all pointed spaces $X = \Sigma^N X'$ whose suspension spectra are $K$-locally strongly dualizable, with $E^*X$ flat as an $E_*$-module,
there is an isomorphism of $\TT$-algebras mod $\mf{m}$:
$$ E_*S_K^{QX_+} \cong_{\mf{m}} \widehat{\TT} \td{E}^*X. $$
\end{lem}

The mod $\mf{m}$ isomorphism in Lemma~\ref{lem:keylemma} is given by the sequence of mod $\mf{m}$ isomorphisms:
\begin{align*}
E_*S_K^{QX_+} & \cong E_* S_K^{\PP (X)^{\mr{mult}}} & & \text{(Lemma~\ref{lem:diagonal})} \\
& \cong_{\mf{m}} E^* \PP (X)^{\mr{mult}} && \text{(Lemma~\ref{lem:product})}\\
& \cong_{\mf{m}} E^* \PP (X)^{\mr{add}} && \text{(Lemma~\ref{lem:exponential})}\\
& \cong_{\mf{m}} E_* S_K^{\PP (X)^{\mr{add}}} && \text{(Lemma~\ref{lem:product})} \\
& \cong E_* \widehat{\PP}_{S_K} (S^{X}_K) && \text{(Lemma~\ref{lem:norm})} \\
& \cong_{\mf{m}} \pi_* \widehat{\PP}_{E} (E^X) && \text{(Lemma~\ref{lem:product})} \\
& \cong_{\mf{m}} \widehat{\TT} \td{E}^*X. && \text{(Lemma~\ref{eq:completealgmodel})}
\end{align*}
Here, $\PP(X)^{\mr{add}}$ and $\PP(X)^{\mr{mult}}$ denote two different $H_\infty$-coalgebra structures on $\PP(X)$, which will be defined in the next section.

Many parts of this appendix apply more generally to aspects of the $H_\infty$-$R$-algebra structure of $R^{QX_+}$, where $R$ is a commutative $S$-algebra, and $X$ is a connected pointed space such that $R \wedge X$ is strongly dualizable as an $R$-module.  Therefore, we shall always implicitly assume $X$ and $R$ satisfy these hypotheses throughout this appendix. At times we shall have to specifically take $R$ to be $S_K$ or $E$.  Observe that the dualizability of $X$ implies that the natural map
$$ (R^X)^{\wedge_R i} \rightarrow R^{(X^{\wedge i})} $$
is an equivalence.  We therefore may simply use $R^{X^i}$ to unambiguously refer to either of these equivalent spectra.

\subsection*{$H_\infty$ coalgebras in spectra}

By an $H_\infty$ coalgebra $C$ in spectra, we shall mean a spectrum $C$ equipped with $\Sigma_k$-equivariant comultiplication maps:
\begin{gather*}
\psi_k: C \rightarrow C^{\wedge k}
\end{gather*}
for all $k \ge 0$, such that for all $k$ and $\ell$ the diagram
$$
\xymatrix{
C \ar[r]^{\psi_2} \ar[dr]_{\psi_{k+\ell}} &
C^{\wedge 2} \ar[d]^{\psi_k \wedge \psi_\ell} \\
& C^{\wedge k+\ell}  
}
$$
commutes in $\mr{Ho}(\Sp_{\Sigma_k \times \Sigma_\ell})$, and 
$$
\xymatrix{
C \ar[r]^{\psi_k} \ar[d]_{\psi_{k\ell}} &
C^{\wedge k} \ar[d]^{(\psi_\ell)^{k}} \\
C^{\wedge kl} \ar[r]_{\cong} &
(C^{\wedge \ell})^{\wedge k}
}
$$
commutes in $\mr{Ho}(\Sp_{\Sigma_k\wr\Sigma_\ell})$.

The $H_\infty$ coalgebra structures $\PP(X)^{\mr{add}}$ and $\PP(X)^{\mr{mult}}$ on $\PP(X)$ will be encoded in structure maps
$$ \psi_k: \PP(X) \rightarrow \PP(X)^k $$
where the maps $\psi_k$ are maps of $E_\infty$ ring spectra (thus making $\PP(X)$ some kind of bialgebra).
The structure maps $\psi_k$ of such $H_\infty$-coalgebra structures on $\PP(X)$ are determined by their restrictions to $\Sigma^\infty X$, given by the composites
$$ \psi_k | X : \Sigma^{\infty}X \hookrightarrow \PP(X) \xrightarrow{\psi_k} \PP(X)^{\wedge k}. $$

For a spectrum $Y$, the zig-zag
$$ Y \xrightarrow{\Delta} \prod_{i = 1}^k Y \xleftarrow{\simeq} \bigvee_{i=1}^k Y $$
determines a canonical map
$$ w_k: Y \rightarrow \bigvee_{i = 1}^k Y $$
in $\mr{Ho}(\Sp_{\Sigma_k})$, where $\Sigma_k$ acts trivially on $Y$, and by permuting the wedge factors of $\bigvee_{i = 1}^k Y$.

The coproducts $\psi_k^{\mr{add}}$ giving rise to $\PP(X)^{\mr{add}}$ are determined (in the sense described above) by the composites
$$ \Sigma^\infty X \xrightarrow{w_k} \bigvee_{i = 1}^k \Sigma^\infty X_i \rightarrow \PP(X)^{\wedge k} $$
where $X_i = X$, regarded as a wedge summand of the $i$th term of the smash product $\PP(X)^{k}$.
The coproducts $\psi_k^{\mr{mult}}$ giving rise to $\PP(X)^{\mr{mult}}$ are determined by the composites
$$ \Sigma^\infty X \xrightarrow{w_{2^k-1}} \bigvee_{\emptyset \ne S \subseteq \ul{k}} \Sigma^\infty X 
\xrightarrow{\bigvee \Delta^S} 
\bigvee_{\emptyset \ne S \subseteq \ul{k}}   
\bigwedge_{s \in S}^k \Sigma^\infty X_s
 \rightarrow \PP(X)^{\wedge k}, $$
where the wedge ranges over non-empty subsets of $\ul{k} = \{1, \ldots, k\}$, and $\Delta^S$ is the $S$-fold diagonal.  

\begin{rmk}\label{rmk:addmult}
We explain why we call these the ``additive'' and ``multiplicative'' coalgebra structures.
Consider first the additive formal group, represented by $\ZZ[[x]]$, with coproduct given by
\begin{gather*}
\psi^{\mr{add}}: \ZZ[[x]] \rightarrow \ZZ[[x_1, x_2]], \\
x \mapsto x_1 + x_2.
\end{gather*}
The $k$-fold coproduct
$$ \psi_k^{\mr{add}}: \ZZ[[x]] \rightarrow \ZZ[[x_1, \ldots, x_k]] $$
is then given by
$$ \psi^{\mr{add}}_k(x) = x_1 + \cdots + x_k. $$

Consider now the multiplicative formal group, again represented by $\ZZ[[x]]$, but now with coproduct
$$ \psi^{\mr{mult}}(x) = x_1 + x_2 + x_1x_2. $$
The $k$-fold coproduct is then given by
$$ \psi_k^{\mr{mult}}(x) = \sum_{\emptyset \ne S \subseteq \ul{k}} \: \prod_{s \in S} x_s. $$
\end{rmk}

\begin{lem}\label{lem:Hinfty}
Suppose $R$ is an $H_\infty$ ring spectrum, and $C$ is an $H_\infty$ coalgebra in spectra.  Then $R^C$ inherits an $H_\infty$-$R$-algebra structure.
\end{lem}

\begin{proof}
The $H_\infty$-$R$-algebra structure of $R^{C}$ is given by structure maps
$$ \xi_k : (R^{C})_{h\Sigma_k}^{\wedge_R k} \rightarrow R^{C} $$
whose adjoints are given by the composites (see, e.g. \cite[Lem.~II.3.3]{Hinfty})
\begin{multline*}
\td{\xi}_k : (R^{C})_{h\Sigma_k}^{\wedge_R k} \wedge C \simeq \left( (R^{C})^{\wedge_R k} \wedge C \right)_{h\Sigma_k} 
\\
\xrightarrow{1 \wedge \psi_k} \left( (R^{C})^{\wedge_R k} \wedge C^{\wedge k} \right)_{h\Sigma_k} \xrightarrow{\mr{ev}^{\wedge k}} R^{\wedge_R k}_{h\Sigma_k} \xrightarrow{\mu_k} R. \end{multline*}
Here $\mu_k$ comes from the $H_\infty$-$R$-algebra structure of $R$ itself: under the isomorphism $R^{\wedge_R k} \cong R$ the composite
$$ R_{h\Sigma_k} = R^{\wedge_R k}_{h\Sigma_k} \xrightarrow{\mu_k} R $$
is the restriction coming from the map of groups $\Sigma_k \rightarrow 1$.  Therefore the map $\td{\xi}_k$ is also given by the composite
\begin{multline*}
\td{\xi}_k : (R^{C})_{h\Sigma_k}^{\wedge_R k} \wedge C \rightarrow 
(R^{C^{\wedge k}})_{h\Sigma_k} \wedge C
\simeq \left( R^{C^{\wedge k}} \wedge C \right)_{h\Sigma_k} 
\\
\xrightarrow{1 \wedge \psi_k} \left( R^{C^{\wedge k}} \wedge C^{\wedge k} \right)_{h\Sigma_k} \xrightarrow{\mr{ev}} R_{h\Sigma_k} \xrightarrow{\Res_{\Sigma_k}^1} R. 
\end{multline*}
\end{proof}

\subsection*{The coalgebra structure of $\Sigma^\infty QX_+$}

The $H_\infty$ structure of $R^{QX_+}$ comes from the cocommutative coalgebra structure on $\Sigma^\infty QX_+$ associated to the diagonal map
$$ QX_+ \xrightarrow{\Delta} (QX_+)^{\wedge 2}. $$
Understanding how this diagonal map interacts with the Kahn splitting is key to everything else in this appendix.  This was worked out by Kuhn in \cite{Kuhndiagonal}; because our language and setting differs somewhat from his, we recall some details.

Recall the convenient point-set level description of the Kahn stable splitting of $QX_+$ given in \cite{KuhnAQG}.

\begin{lem}[\cite{KuhnAQG}]
The map
$$ s_X : \PP(X) \rightarrow \Sigma^\infty QX_+ $$
of $E_\infty$ ring spectra adjoint to the natural inclusion of spectra
$$ \Sigma^\infty X \rightarrow \Sigma^\infty QX_+ $$
is a weak equivalence when $X$ is connected.
\end{lem}

\begin{lem}[\cite{Kuhndiagonal}]\label{lem:diagonal}
The equivalence 
$$ s_X : \PP(X)^{\mr{mult}} \xrightarrow{\simeq} \Sigma^\infty QX_+ $$
is a map $H_\infty$ coalgebras.
\end{lem}

\begin{proof}
We need to show that the following diagram commutes:
$$
\xymatrix{
\PP(\Sigma^\infty X) \ar[d]_{s_X}^\simeq \ar[r]^-{\psi_k^{\mr{mult}}} & \PP(\Sigma^\infty X)^{\wedge k} 
\ar[d]^{s_{X}^{\wedge k}}_\simeq
\\
\Sigma^\infty QX_+ \ar[r]_-{\Delta} 
& \Sigma^\infty (QX_+)^{\wedge k}
} $$
By adjointness, this follows from the commutativity of the diagram:
$$
\xymatrix{
\Sigma^\infty X \ar@{^{(}->}[dd] \ar[rrr]^{\psi_k^{\mr{mult}}|X} \ar@{^{(}->}[rd] &&& \PP(\Sigma^\infty X)^{\wedge k} 
\ar[dd]^{s_X^{\wedge k}}
\\
& \Sigma^\infty X_+ \ar[r]_-{\Delta} \ar@{^{(}->}[dl] 
&\Sigma^\infty (X_+)^{\wedge k} \ar@{^{(}->}[ur] \ar@{^{(}->}[dr] 
\\
\Sigma^\infty QX_+ \ar[rrr]_{\Delta} 
&&& \Sigma^\infty (QX_+)^{\wedge k}
} $$
\end{proof}

\subsection*{$E$-homology of products}

Disregarding multiplicative structure, for $X$ connected, the following $K$-local spectra are all equivalent:
\begin{equation}\label{eq:SKQX}
S_K^{QX_+} \simeq S_K^{\PP(X)} \simeq \prod_i (S^{X^{i}}_K)^{h\Sigma_i}  \simeq \prod_i [(S^{X^{i}}_K)_{h\Sigma_i}]_K \simeq \widehat{\PP}_{S_K}(S^X_K)_K
\end{equation}
(where the second to last equivalence is given by the product of norm maps, and the last equivalence is by \cite[Cor.~6.1.3]{BehrensDavis}).  We would like to study multiplicative structures on these equivalent spectra by means of $E$-homology, but this is complicated by the fact that homology, in general, does not commute with products.  The following lemma indicates that in our setting, there is an instance where completed Morava $E$-homology does commute with a certain infinite product.

We shall say that a $K$-local spectrum $Y$ is \emph{strongly dualizable} if the natural map
$$  S_K^Y\wedge Y \rightarrow Y^Y $$
is a $K$-local equivalence.

\begin{lem}\label{lem:product}
Suppose that $Y$ is a strongly dualizable $K$-local spectrum, such that $E_* Y$ is flat.  Then the natural map
$$ E_* \widehat{\PP}_{S_K} Y \rightarrow \pi_* \widehat{\PP}_E (E \wedge_{S_K} Y) = \widehat{\TT}(E_* Y) $$
is an isomorphism.
\end{lem}

As in this paper we are working under the convention that $E_*$ denotes completed Morava $E$-homology, we shall let $E_*^{uc}$ denote \emph{uncompleted} Morava $E$-homology. 
Lemma~\ref{lem:product} will be proven using a variant of a spectral sequence of Hovey \cite{Hovey} for the homology of a product:
\begin{equation}\label{eq:HSS}
E_2^{s,*} = R^s\prod_{E^{uc}_*E} E^{uc}_* X_\alpha \Rightarrow E^{uc}_{*-s} \prod X_\alpha.
\end{equation}
 In (\ref{eq:HSS}), the $E_2$-term is given in terms of derived functors of the product of $E^{uc}_*E$-comodules (which in general is different from the product of sets/modules), and Hovey proves that if the spectra $X_\alpha$ are $E$-local, then the spectral sequence converges strongly.  

Note that if $I \subset E_*$ is an invariant regular ideal, the Hopf algebroid structure on $(E_*, E^{uc}_*E)$ descends to the quotient $(E_*/I, E_*E/I)$.  Let $E/I$ denote the associated spectrum.  We will need to use the following slight variant of spectral sequence (\ref{eq:HSS}).  

\begin{lem}\label{lem:HSS}
For a set of spectra $\{X_\alpha \}$ with $E^{uc}_*X_\alpha$ flat over $E_*$, there is a spectral sequence
$$ 
E_2^{s,*} = R^s\prod_{E_*E/I} (E/I)_* X_\alpha \Rightarrow (E/I)_{*-s} \prod X_\alpha.
$$
If all of the spectra $X_\alpha$ are $E$-local, then this spectral sequence converges strongly.
\end{lem}

\begin{proof}
Hovey constructed spectral sequence (\ref{eq:HSS}) from the exact couple obtained by taking the $E$-homology of the product of modified $E$-based Adams resolutions of each of the $X_\alpha$.  The desired spectral sequence is obtained by instead taking the $E/I$-homology of this product of $E$-based Adams resolutions.
This spectral sequences converges strongly if each of the $X_\alpha$ are $E$-local, for the same reasons the original one did.
\end{proof}

The $E_2$-term of this spectral sequence is in general quite mysterious, but Hovey does prove two key facts about it.
\begin{itemize}
\item If $\{N_\alpha\}_\alpha$ is a collection of $E_*/I$-modules, then the derived functors of the product of the extended comodules $E_*E/I \otimes_{E_*/I} N_\alpha$ are computed to be
$$
R^s \prod_{E_*E/I} E_*E/I \otimes_{E_*/I} N_\alpha \cong 
\begin{cases}
E_*E/I \otimes_{E_*/I} \prod N_\alpha, & s = 0, \\
0, & s > 0.
\end{cases}
$$

\item If $\{M_\alpha\}_\alpha$ is a collection of $E_*E/I$-comodules, and 
$$ M_\alpha \rightarrow J_\alpha^0 \rightarrow J_\alpha^1 \rightarrow \cdots $$
are resolutions of each of these comodules by extended comodules, then the derived functors of product may be computed as
$$
R^s \prod_{E_*E/I} M_\alpha \cong H^s\left(\prod_{E_*E/I}J^*_\alpha\right).
$$
\end{itemize}

%\begin{lem}
%Suppose that $\{M_\alpha\}$ is a collection of $E^{uc}_*E$-comodules, and suppose that an invariant ideal $I \subset E_*$ annihilates each of the $M_\alpha$.  Then the derived functors of the product of $E_*E$-comodules are isomorphic to the derived functors of the product of $E_*E/I$-comodules:
%$$ R^s \prod_{E_*E} M_\alpha \cong R^s \prod_{E_*E/I} M_\alpha. $$
%\end{lem}

%\begin{proof}
%Observe that if $M$ is a $E_*E$-comodule annihilated by $I$, then there is an isomorphism
%$$ E_*E \otimes_{E_*} M \cong E_*E/I \otimes_{E_*/I} M. $$ 
%It follows that there is an adjunction
%$$ - \otimes_{E_*} E_*/I : \mr{CoMod}_{E_*E} \leftrightarrows \mr{CoMod}_{E_*E/I}: \mc{U} $$
%(where $\mc{U}$ is the restriction), and in particular there is an isomorphism 
%$$ \prod_{E_*E} M_\alpha \cong \prod_{E_*E/I} M_\alpha. $$
%Consider the cobar resolution of $M_\alpha$: we have
%$$ C^*(E_*E, E_*E, M_\alpha) \cong C^*(E_*E/I, E_*E/I, M_\alpha ) $$
%and thus
%$$ \prod_{E_*E}C^*(E_*E, E_*E, M_\alpha) \cong \prod_{E_*E/I}C^*(E_*E/I, E_*E/I, M_\alpha ). $$
%Taking cohomology, we deduce that there is an isomorphism
%$$ R^s \prod_{E_*E} M_\alpha \cong R^s \prod_{E_*E/I} M_\alpha. $$
%\end{proof}

We now assume the invariant ideal $I$ is of the form $(p^{i_0}, v_1^{i_1}, \ldots, v_{h-1}^{i_{h-1}})$.
Let $\MS$ denote the extended Morava stabilizer group.  Then we have \cite{HoveyEop}
$$ E_*E/I \cong \Map^c(\MS, E_*/I). $$
An $E_*E/I$-comodule structure on an $E_*/I$-module $M$ is the same thing as a (twisted $E_*/I$-linear) continuous action of $\MS$ on $M$.  Here, $M$ is given the discrete topology, so that continuity of the action is equivalent to the statement that every element $m \in M$ has an open stabilizer $\mr{Sta}_\MS(m)$.   For a collection $\{M_\alpha\}$ of $E_*E/I$-comodules, the product is easily seen to be
$$ \prod_{E_*E/I} M_\alpha = \{ (m_\alpha) \in \prod M_\alpha \: : \: \bigcap_\alpha \mr{Sta}_\MS(m_\alpha) \: \text{is open} \}. $$
In particular, if there is a fixed open subgroup $U \le \MS$ (independent of $\alpha$) such that the $\MS$-action on each of the modules $M_\alpha$ restricts to the trivial action on $U$, then the product in $E_*E/I$-comodules agrees with the ordinary product:
$$ \prod_{E_*E/I} M_\alpha = \prod M_\alpha. $$ 
Less obviously, the following lemma holds, whose proof seems to require cohomological finiteness properties of $\MS$.

\begin{lem}\label{lem:HSSE2}
Let $\{M_\alpha\}$ be a collection of $E_*E/I$-comodules, and suppose that 
there is a fixed open subgroup $U \le \MS$ (independent of $\alpha$) such that the $\MS$-action on each of the modules $M_\alpha$ restricts to the trivial action on $U$.  Then 
$$ 
R^s\prod_{E_*E/I} M_\alpha = \begin{cases}
\prod M_\alpha, & s = 0, \\
0, & s > 0. 
\end{cases}
$$ 
\end{lem}

\begin{proof}
By \cite[Thm.~5.1.2]{SymondsWeigel}, there is a resolution of the trivial $\MS$-module $\ZZ_p$ by finitely generated free $\ZZ_p[[\MS]]$-modules:
$$ \ZZ_p \leftarrow P_0 \leftarrow P_1 \leftarrow P_2 \leftarrow \cdots. $$
Write
$$ P_i = \ZZ_p[[\MS]] \otimes_{\ZZ_p} N_i $$
where $N_i$ are finite free $\ZZ_p$-modules.
Then for any $E_*E/I$-comodule $M$, we get an induced resolution by extended comodules:
$$ M \rightarrow E_*E/I \otimes_{E_*/I} \Hom_{\ZZ_p}(N_0, M) \rightarrow E_*E/I \otimes_{E_*/I} \Hom_{\ZZ_p}(N_1, M) \rightarrow \cdots. $$
(Note that without the finiteness conditions on the modules $N_i$, the $\Hom$'s above would have to be replaced with $\Hom^c$'s --- continuous homomorphisms; the argument that follows would fail in this more general context.)
Denote this resolution
$$ C^*(M;P_*) = (E_*E/I \otimes_{E_*/I} \Hom_{\ZZ_p}(N_0, M) \rightarrow E_*E/I \otimes_{E_*/I} \Hom_{\ZZ_p}(N_1, M) \rightarrow \cdots ) $$
to emphasize its functoriality in $E_*E/I$-comodules $M$.  Since our hypotheses ensure $\prod M_\alpha$ is an $E_*E/I$-comodule, the isomorphisms
$$ \prod_{E_*E/I} E_*E/I \otimes_{E_*/I} \Hom(N_i, M_\alpha) \cong E_*E/I \otimes_{E_*/I}  \Hom(N_i, \prod M_\alpha) $$
extend to give isomorphisms
$$ \prod_{E_*E/I} C^*(M_\alpha; P_*) \cong C^*\left(\prod M_\alpha; P_*\right). $$
We therefore have
\begin{align*}
R^s\prod_{E_*E/I} M_\alpha 
& \cong H^s \left( \prod_{E_*E/I} C^*(M_\alpha; P_*) \right) \\
& \cong H^s \left(  C^*\left(\prod M_\alpha; P_*\right) \right)\\
& \cong \begin{cases}
\prod M_\alpha, & s = 0, \\
0, & s > 0.
\end{cases}
\end{align*}
\end{proof}

\begin{proof}[Proof of Lemma~\ref{lem:product}]
Let $I = (p^{i_0}, \ldots, v_{h-1}^{i_{h-1}}) \subset E_*$ be an invariant ideal.  By Lemma~\ref{lem:HSS}, there is a convergent spectral sequence
$$ 
E_2^{s,*} = R^s\prod_{E_*E/I} (E/I)_* Y^i_{h\Sigma_i} \Rightarrow (E/I)_{*-s} \widehat{\PP}_{S_K}Y.
$$
Moreover, since $Y$ is strongly dualizable, $(E/I)_*Y$ is finite, and there is therefore an open subgroup $U \le \MS$ which acts trivially on $(E/I)_*Y$.  Since $\Delta^*[1]$ is finitely generated as an $E_*$-module (and $\Delta^*[1]$ generates $\Delta^*$), it follows that there is an open subgroup $U' \le U$ which acts trivially on
$$ \left(\mr{Sym}_{E_*}(\Delta^* \otimes_{E_*} E_*Y)\right)/I \cong \bigoplus_{i} (\TT\bra{i} E_*Y)/I. $$
In particular, $U'$ acts trivially on 
$$ (E/I)_*Y^i_{h\Sigma_i} = (\TT\bra{i} E_*Y)/I. $$
Using Lemma~\ref{lem:HSSE2}, we conclude that the natural map
$$ (E/I)_* \prod_i Y^i_{h\Sigma_i} \rightarrow \prod_i (\TT\bra{i} E_*Y)/I $$
is an isomorphism.  Viewed as an inverse system indexed on $I$, the above system is Mittag-Leffler.  Taking inverse limits over $I$ (and using the facts that $E_*$ is Noetherian, and $\TT\bra{i} E_*Y$ is flat over $E_*$), we therefore obtain an isomorphism
$$ E_* \widehat{\PP}_{S_K}Y \cong \varprojlim_I (E/I)_* \prod_i Y^i_{h\Sigma_i} \cong \widehat{\TT}E_*Y. $$
\end{proof}

\begin{cor}\label{cor:flat}
Suppose that $Y$ is a strongly dualizable $K$-local spectrum, such that $E_* Y$ is flat. Then 
$E_* \widehat{\PP}_{S_K}Y$ is flat. 
\end{cor}

\begin{proof}
Since $E_*$ is Noetherian, it is coherent, and hence products of flat $E_*$-modules are flat \cite{Chase}.
\end{proof}

\subsection*{The $H_\infty$ structure of $R^{\PP(X)^{\mr{add}}}$}

%The goal of this subsection is to describe the $H_\infty$ structure maps
%$$ \xi_k: (R^{\PP(X)^{\mr{add}}})^{\wedge_R k}_{\Sigma_k} \rightarrow R^{\PP(X)^{\mr{add}}}.$$
 
We return to the motivating analogy (see Remark~\ref{rmk:addmult}) of the bialgebra representing the additive
coproduct.  The additive coproduct $\psi^{\mr{add}}_k$ on $\ZZ[[x]]$ satisfies
\begin{align*}
\psi^{\mr{add}}_k(x^i) & = (x_1 + \cdots + x_k)^i \\
& = \sum_{\substack{I = (i_1, \ldots, i_k) \\ \norm{I} = i}} \frac{i!}{i_1!\cdots i_k!}x_1^{i_1}\cdots x_k^{i_k}
\end{align*}
where, for a sequence $I = (i_1, \ldots, i_k)$ of non-negative integers, we define 
$$\norm{I} := i_1 + \cdots + i_k.$$ 

We wish to give a homotopy theoretic refinement of the above formula in the case of $\PP(X)^{\mr{add}}$.
Define
$$ \Sigma_I := \Sigma_{i_1} \times \cdots \times \Sigma_{i_k}, $$
and let $\Sigma_{(I)}$ denote the subgroup of $\Sigma_k$ which preserves the sequence $I$, and define $\Sigma_{[I]}$ to be the subgroup of $\Sigma_i$ given by
$$ \Sigma_{[I]} := \Sigma_{(I)} \ltimes \Sigma_{I}. $$
%There are $\Sigma_k$-equivariant equivalences
%\begin{equation}\label{eq:Pwedge}
%\PP(Y^{\vee k}) \simeq \PP(Y)^{\wedge k} \simeq \bigvee_{I = (i_1, \ldots, i_k)} (\Sigma_{k})_+ \wedge_{\Sigma_{(I)}} Y^{\norm{I}}_{h\Sigma_I}. 
%\end{equation}
%Let $\alpha_I$ denote the $I$-component of the above equivalence
%$$ \alpha_I : (Y^{\vee k})^{\wedge \norm{I}}_{h\Sigma_{\norm{I}}} \rightarrow (\Sigma_k)_+ \wedge_{\Sigma_{(I)}}  Y^{\wedge \norm{I}}_{h\Sigma_{I}}. $$
%Since there is a $\Sigma_k$-equivariant equivalence
%$$ (\Sigma_k)_+ \wedge_{\Sigma_{(I)}}  Y^{\wedge \norm{I}}_{h\Sigma_{I}} = \mr{Ind}^{\Sigma_k}_{\Sigma_{(I)}} Y^{\wedge \norm{I}}_{h\Sigma_{I}} \xrightarrow{\simeq} \mr{CoInd}^{\Sigma_k}_{\Sigma_{(I)}} Y^{\wedge \norm{I}}_{h\Sigma_{I}}. $$
%the $\Sigma_k$-equivariant map $\alpha_I$ determines and is determined by a $\Sigma_{(I)}$-equivariant map
%$$ \td{\alpha}_{I}: (Y^{\vee k})^{\wedge \norm{I}}_{h\Sigma_{\norm{I}}} \rightarrow Y^{\wedge \norm{I}}_{h\Sigma_{I}}. $$

%The following is a consequence of \cite[VII.1.10]{LMN}.

%\begin{lem}\label{lem:diagonaltransfer}
%The composite 
%$$ Y^{\wedge \norm{I}}_{h\Sigma_{\norm{I}}} \xrightarrow{\Delta} (Y^{\vee k})^{\wedge \norm{I}}_{h\Sigma_{\norm{I}}} \xrightarrow{\td{\alpha}_{I}} Y^{\wedge \norm{I}}_{h\Sigma_{I}} $$
%is equal to the transfer $\Tr^{\Sigma_{\norm{I}}}_{\Sigma_I}$ in $\mr{Ho}(\Sp_{\Sigma_{(I)}})$.  
%\end{lem}

The $H_\infty$-$R$-algebra structure of $R^{\PP(X)^{\mr{add}}}$ is given by structure maps
$$ \xi^{\mr{add}}_k : (R^{\PP(X)^{\mr{add}}})_{h\Sigma_k}^{\wedge_R k} \rightarrow R^{\PP(X)^{\mr{add}}} $$
whose adjoints are given by the composites (see Lemma~\ref{lem:Hinfty})
\begin{multline*}
\td{\xi}^{\mr{add}}_k : (R^{\PP(X)})_{h\Sigma_k}^{\wedge_R k} \wedge \PP(X) \rightarrow 
(R^{(\PP(X)^{\wedge k})})_{h\Sigma_k} \wedge \PP(X)
\simeq \left( R^{(\PP(X)^{\wedge k})} \wedge \PP(X) \right)_{h\Sigma_k} 
\\
\xrightarrow{1 \wedge \psi^{\mr{add}}_k} \left( R^{(\PP(X)^{\wedge k})} \wedge \PP(X)^{\wedge k} \right)_{h\Sigma_k} \xrightarrow{\mr{ev}} R_{h\Sigma_k} \xrightarrow{\Res_{\Sigma_k}^1} R. 
\end{multline*}
There are $\Sigma_k$-equivariant equivalences
\begin{equation}\label{eq:Pwedge}
\PP(X)^{\wedge k} \simeq \bigvee_{I = (i_1, \ldots, i_k)} X^{\norm{I}}_{h\Sigma_I} 
\end{equation}
where $\Sigma_k$ acts on the indexing set by permuting the sequences.
It follows that there are equivalences   
$$ R^{(\PP(X)^{\wedge k})} 
\simeq \prod_i \bigvee_{\substack{I  = (i_1, \ldots, i_k) \\ \norm{I} = i}}  R^{X_{h\Sigma_I}^{\norm{I}}}. $$
Note that there is are equivalences
$$ 
\left( \bigvee_{\substack{I  = (i_1, \ldots, i_k) \\ \norm{I} = i}}  R^{X_{h\Sigma_I}^{\norm{I}}}\right)_{h\Sigma_k}
\simeq
\bigvee_{[I] \in \mc{I}^k_i}  \left( \mr{Ind}_{\Sigma_{(I)}}^{\Sigma_k} R^{X_{h\Sigma_I}^{\norm{I}}}\right)_{h\Sigma_k} \simeq 
\bigvee_{[I] \in \mc{I}^k_i} \left( R^{X_{h\Sigma_I}^{\norm{I}}} \right)_{h\Sigma_{(I)}}
$$
where $\mc{I}^k_i$ is the set of $\Sigma_k$-orbits:
$$ \mc{I}^k_i := \{ I = (i_1, \ldots, i_k) \: : \: \norm{I} = i \}/\Sigma_k. $$

\begin{lem}\label{lem:xiadd}
The map $\xi^{\mr{add}}_k$ is given by the composite 
\begin{align*}
\xi_k^{\mr{add}} : 
(R^{\PP(X)})_{h\Sigma_k}^{\wedge_R k} & 
\rightarrow (R^{(\PP(X)^{\wedge k})})_{h\Sigma_k} \\
& \simeq \left( \prod_i \bigvee_{\substack{I  = (i_1, \ldots, i_k) \\ \norm{I} = i}} R^{X^{\norm{I}}_{h\Sigma_I}} \right)_{h\Sigma_k} \\
& \rightarrow \prod_i \left( \bigvee_{\substack{I  = (i_1, \ldots, i_k) \\ \norm{I} = i}} R^{X^{\norm{I}}_{h\Sigma_I}} \right)_{h\Sigma_k} \\
& \simeq \prod_i \bigvee_{[I]  \in \mc{I}^k_i} \left(  R^{X^{\norm{I}}_{h\Sigma_I}} \right)_{h\Sigma_{(I)}}  \\
& \xrightarrow{(\xi^{\mr{add}}_{I,i})_{I,i}} \prod_i R^{X^{i}_{h\Sigma_i}} \\
& \simeq R^{\PP (X)}.
\end{align*}
where the only non-zero matrix coefficients 
$$ \xi^{\mr{add}}_{I,i}: (R^{X^{\norm{I}}_{h\Sigma_I}})_{h\Sigma_{(I)}} 
\rightarrow R^{X^{i}_{h\Sigma_i}} $$
occur when $i = \norm{I}$, for which they are adjoint to the composites
\begin{multline*}
\td{\xi}^{\mr{add}}_{I,i} : (R^{X^{\norm{I}}_{h\Sigma_I}})_{h\Sigma_{(I)}} \wedge X^{i}_{h\Sigma_{\norm{I}}} \simeq
(R^{X^{\norm{I}}_{h\Sigma_I}} \wedge X^{\norm{I}}_{h\Sigma_{\norm{I}}} )_{h\Sigma_{(I)}} 
\\
\xrightarrow{1 \wedge \mr{Tr}^{\Sigma_{\norm{I}}}_{\Sigma_I}}  
(R^{X^{\norm{I}}_{h\Sigma_I}} \wedge X^{\norm{I}}_{h\Sigma_{I}} )_{h\Sigma_{(I)}}
\xrightarrow{\mr{ev}} R_{h\Sigma_{(I)}} \xrightarrow{\mr{Res}^{1}_{\Sigma_{(I)}}} R.
\end{multline*}
\end{lem}

\begin{proof}
Using (\ref{eq:Pwedge}), the map $\td{\xi}^{\mr{add}}_k$ is given by the composite
\begin{multline*}
(R^{\PP(X)})_{h\Sigma_k}^{\wedge_R k} \wedge \PP(X) \rightarrow 
\left( R^{(\PP(X)^{\wedge k})} \wedge \PP(X) \right)_{h\Sigma_k}
\simeq \left(R^{\bigvee_{I}  X^{\norm{I}}_{h\Sigma_I}} \wedge \bigvee_i X^i_{h\Sigma_i} \right)_{h\Sigma_k} \\
\xrightarrow{1 \wedge \bigvee_i \sum_I  \mr{Tr}_{\Sigma_I}^{\Sigma_i}} 
\left(R^{\bigvee_{I}  X^{\norm{I}}_{h\Sigma_I}} \wedge \bigvee_i \bigvee_{\norm{I} = i} X^i_{h\Sigma_I} \right)_{h\Sigma_k}
 \xrightarrow{\mr{ev}} R_{h\Sigma_k} \xrightarrow{\Res_{\Sigma_k}^1} R.
\end{multline*}
The result follows.
\end{proof}

\subsection*{The $H_\infty$ structure of $R^{\PP(X)^{\mr{mult}}}$}

We now turn to a similar analysis of the $H_\infty$ structure of $R^{\PP(X)^{\mr{mult}}}$.  This is essentially done in \cite{Kuhndiagonal}, except that there the computation is of the product, rather than the full $H_\infty$ structure.  In order to facilitate the understandability of the combinatorics, we again return to the motivating analogy (see Remark~\ref{rmk:addmult}) of the bialgebras representing the multiplicative groups.
We compute the multiplicative coproduct $\psi^{\mr{mult}}_k$ on $\ZZ[[x]]$ to be
\begin{align*}
\psi^{\mr{mult}}_k (x^i) & = \left( \sum_{\emptyset \ne S \subseteq \ul{k}} \: \prod_{s \in S} x_s \right)^i \\
& = \sum_{\substack{J = (j_S) \\ \norm{J} = i}} \frac{i!}{\prod_S j_S!} \prod_{S} \prod_{s \in S} x_s^{i_S}
\end{align*}
where the last sum now ranges over sequences 
$$ J = (j_S \: : \: \emptyset \ne S \subseteq \ul{k}) $$
of non-negative integers indexed by subsets of $\ul{k}$, and 
$$ \norm{J} := \sum_{S} j_S. $$
Note that the exponent $i_s$ of $x_s$ in the $J$th summand of $\psi_k^{\mr{mult}}(x^i)$ is given by
$$ i_s = \sum_{\substack{S \\ s \in S}} i_S. $$
We therefore define $I(J)$ to be the sequence
$$ I(J) = (i_1, \ldots, i_k) $$
with $i_s$ defined as above.  Note that 
$$ \norm{I(J)} = \sum_S \abs{S} j_S. $$

The homotopical refinement of the formula for $\psi^{\mr{mult}}_k(x^i)$ proceeds as follows.
For such a sequence $J = (j_S)$ of non-negative integers, define 
$$ \Sigma_J := \prod_S \Sigma_{j_S}. $$
Let $\Sigma_{(J)}$ denote the subgroup of $\Sigma_k$ which preserves the sequence $J$:
$$ \Sigma_{(J)} = \{ \sigma \in \Sigma_k \: : \: j_S = j_{\sigma(S)} \} $$
and define $\Sigma_{[J]}$ to be the subgroup of $\Sigma_{\norm{J}}$ given by
$$ \Sigma_{[J]} := \Sigma_{(J)} \ltimes \Sigma_{J}. $$

\begin{lem}\label{lem:ximult}
The $H_\infty$ structure maps $\xi^{\mr{mult}}_k$ for $R^{\PP(X)^{\mr{mult}}}$ are given by the composites 
\begin{align*}
\xi_k^{\mr{mult}} : 
(R^{\PP(X)})_{h\Sigma_k}^{\wedge_R k} & 
\rightarrow (R^{(\PP(X)^{\wedge k})})_{h\Sigma_k} \\
& \simeq \left( \prod_i \bigvee_{\substack{I  = (i_1, \ldots, i_k)\\ \norm{I} = i}}  R^{X^{\norm{I}}_{h\Sigma_I}} \right)_{h\Sigma_k} \\
& \rightarrow  \prod_i \left( \bigvee_{\substack{I  = (i_1, \ldots, i_k)\\ \norm{I} = i}}  R^{X^{\norm{I}}_{h\Sigma_I}} \right)_{h\Sigma_k} \\
& \simeq \prod_i \bigvee_{[I] \in \mc{I}^k_i } (R^{X^{\norm{I}}_{h\Sigma_I}})_{h\Sigma_{(I)}} \\
& \xrightarrow{(\xi^{\mr{mult}}_{I,i})_{I,i}} \prod_i R^{X^{i}_{h\Sigma_i}} \\
& \simeq R^{\PP (X)}.
\end{align*}
Here the matrix coefficients 
$$ \xi^{\mr{mult}}_{I,i}: (R^{X^{\norm{I}}_{h\Sigma_I}})_{h\Sigma_{(I)}} 
\rightarrow R^{X^{i}_{h\Sigma_i}} $$
are now given by sums
$$ \xi^{\mr{mult}}_{I,i} = \sum_{[J] \in \mc{J}(I,i)} \xi^{\mr{mult}}_{J} $$
where 
$$ \mc{J}(I,i) = \{ J = (j_S) \: : \: \norm{J} = i, I(J) = I \}/\Sigma_{(I)}, $$ 
$\xi^{\mr{mult}}_{J}$ is adjoint to the composite
\begin{align*}
\td{\xi}^{\mr{mult}}_{J} : (R^{X^{\norm{I}}_{h\Sigma_I}})_{h\Sigma_{(I)}} \wedge X^{i}_{h\Sigma_{i}} 
& \simeq (R^{X^{\norm{I}}_{h\Sigma_I}} \wedge X^{i}_{h\Sigma_{i}} )_{h\Sigma_{(I)}} \\
& \xrightarrow{\Tr^{\Sigma_{(I)}}_{\Sigma_{(J)}}}   
(R^{X^{\norm{I}}_{h\Sigma_I}} \wedge X^{i}_{h\Sigma_{i}} )_{h\Sigma_{(J)}} \\
& \xrightarrow{1 \wedge \mr{Tr}^{\Sigma_{i}}_{\Sigma_J}}
(R^{X^{\norm{I}}_{h\Sigma_I}} \wedge X^{i}_{h\Sigma_{J}} )_{h\Sigma_{(J)}} \\
& \xrightarrow{1 \wedge \Delta_J}
(R^{X^{\norm{I}}_{h\Sigma_J}} \wedge X^{\norm{I}}_{h\Sigma_{J}} )_{h\Sigma_{(J)}} \\
& \xrightarrow{1 \wedge \mr{Res}^{\Sigma_{I}}_{\Sigma_{J}}}
(R^{X^{\norm{I}}_{h\Sigma_I}} \wedge X^{\norm{I}}_{h\Sigma_{I}} )_{h\Sigma_{(J)}} \\
& \xrightarrow{\mr{Res}^{\Sigma_{(I)}}_{\Sigma_{(J)}}}
(R^{X^{\norm{I}}_{h\Sigma_I}} \wedge X^{\norm{I}}_{h\Sigma_{I}} )_{h\Sigma_{(I)}} \\
& \xrightarrow{\mr{ev}} R_{h\Sigma_{(I)}} \\
& \xrightarrow{\mr{Res}^{1}_{\Sigma_{(I)}}} R,
\end{align*}
and $\Delta_J$ is induced by the diagonal map
$$ X^i = \bigwedge_S X^{j_S} \xrightarrow{\bigwedge \Delta_{\abs{S}}} \bigwedge_S X^{\abs{S}j_S} = X^{\norm{I}}. $$ 
\end{lem}

\begin{proof}
This result follows from the fact that the map $\td{\xi}^{\mr{mult}}_k$ is given by the composite
\begin{align*}
(R^{\PP(X)})_{h\Sigma_k}^{\wedge_R k} \wedge \PP(X) & 
\rightarrow \left( R^{(\PP(X)^{\wedge k})} \wedge \PP(X) \right)_{h\Sigma_k} \\
& \simeq \left(R^{\bigvee_{I}  X^{\norm{I}}_{h\Sigma_I}} \wedge \bigvee_i X^i_{h\Sigma_i} \right)_{h\Sigma_k} \\
& \xrightarrow{1 \wedge \bigvee_i \sum_{J} \Delta_J  \mr{Tr}_{\Sigma_J}^{\Sigma_i}} 
\left(R^{\bigvee_{I}  X^{\norm{I}}_{h\Sigma_I}} \wedge \bigvee_i \bigvee_{\norm{J} = i} X^{\norm{I(J)}}_{h\Sigma_J} \right)_{h\Sigma_k} \\
& \xrightarrow{1 \wedge \bigvee_i \bigvee_{J} \mr{Res}_{\Sigma_J}^{\Sigma_{I(J)}}} 
\left(R^{\bigvee_{I}  X^{\norm{I}}_{h\Sigma_I}} \wedge \bigvee_I  X^{\norm{I}}_{h\Sigma_I} \right)_{h\Sigma_k} \\
& \xrightarrow{\mr{ev}} R_{h\Sigma_k} \\
& \xrightarrow{\Res_{\Sigma_k}^1} R.
\end{align*}
\end{proof}

\subsection*{The mod $\mf{m}$ $E$-cohomology of $\PP(X)^{\mr{add}}$ and $\PP(X)^{\mr{mult}}$}

We will now argue the crucial point, namely, as long as $X$ is a sufficiently large suspension, the mod $\mf{m}$ $\TT$-algebra structures on $E^* \PP(X)^{\mr{add}}$ and $E^* \PP(X)^{\mr{mult}}$ are indistinguishable.

\begin{lem}\label{lem:exponential}
Suppose that $X = \Sigma^N X'$ is a pointed space whose suspension spectrum is $K$-locally strongly dualizable.
Then for $N \gg 0$ the identity map gives an isomorphism of $\TT$-algebras mod $\mf{m}$:
$$ E^* \PP(X)^{\mr{mult}} \cong_{\mf{m}} E^* \PP(X)^{\mr{add}}. $$
\end{lem}

\begin{proof}
Comparing Lemmas~\ref{lem:xiadd} and Lemmas~\ref{lem:ximult} (with $R = S_K$), the $H_\infty$ structure map $\xi_k^{\mr{add}}$ is the map obtained from $\xi_k^{\mr{mult}}$ by only including the terms $\xi^{\mr{mult}}_J$ for $J = (j_S)$ with $j_S = 0$ for $\abs{S} > 1$.  
Since the $\TT$-algebra structure is completely determined by the maps $\xi^{\mr{mult}}_2$ and $\xi^{\mr{mult}}_p$, we may assume that $k$ is either $2$ or $p$.  It then suffices to show that if $j_S$ is not zero for some $S \subseteq \ul{k}$ with $\abs{S}> 1$, then for $N$ sufficiently large the induced map
$$ \xi^{\mr{mult}}_{J}: E_* \left( \left(S_K^{X^{\norm{I(J)}}_{h\Sigma_{I(J)}}}\right)_{h\Sigma_{(I(J))}}\right)/\mf{m} \rightarrow E_*\left(S_K^{X^{\norm{J}}_{h\Sigma_{\norm{J}}}}\right)/\mf{m}
$$
is zero.  It suffices to show that the map
$$ \Delta^*_J: E_* \left( \left(S_K^{X^{\norm{I(J)}}_{h\Sigma_{J}}}\right)_{h\Sigma_{(J)}}\right)/\mf{m} \rightarrow E_*\left(\left( S_K^{X^{\norm{J}}_{h\Sigma_J}}\right)_{h\Sigma_{(J)}}\right)/\mf{m} $$
is zero,
as the above map factors through this.
Since $X$ is $K$-locally dualizable, and $K$-locally $\Sigma_{(J)}$ homotopy orbits are equivalent to homotopy fixed points, this map is isomorphic to the map
\begin{equation}\label{eq:deltaJ}
 \Delta^*_J: \td{E}^* \left((X^{\norm{I(J)}}_{h\Sigma_J})_{h\Sigma_{(J)}}\right)/\mf{m} \rightarrow \td{E}^* \left((X^{\norm{J}}_{h\Sigma_J})_{h\Sigma_{(J)}}\right)/\mf{m}. 
\end{equation}
The group $\Sigma_{(J)}$ naturally acts by permutation on the representations
\begin{align*}
\rho & = \bigoplus_{\substack{0 \ne S \subseteq \ul{k} \\ j_S \ne 0}} \RR , \\
\rho' & = \bigoplus_{\substack{0 \ne S \subseteq \ul{k} \\ j_S \ne 0}} \RR^{\abs{S}}. 
\end{align*}
The diagonal map
$$ \Delta_J: \rho \rightarrow \rho' $$
is $\Sigma_{(J)}$-equivariant.  Let $\rho^\perp$ be the orthogonal complement of $\rho$ in $\rho'$.  Our assumption on $J$ ensures that $\rho^\perp$ has positive dimension.  Writing $X = \Sigma^N X'$, the map (\ref{eq:deltaJ}) now can be interpreted as a composite
\begin{align*}
& \td{E}^* \left(\left( S^{N\rho + N\rho^\perp} \wedge (S^{\norm{I(J)}-N\abs{\rho'}} \wedge X'^{\norm{I(J)}})_{h\Sigma_J} \right)_{h\Sigma_{(J)}}\right)/\mf{m} \\
& \xrightarrow{1 \wedge \Delta_J} \td{E}^* \left(\left(  S^{N\rho + N\rho^\perp} \wedge (S^{\norm{J}-N\abs{\rho}} \wedge X'^{\norm{J}})_{h\Sigma_J} \right)_{h\Sigma_{(J)}}\right)/\mf{m} 
\\
& \xrightarrow{\Delta_J \wedge 1} \td{E}^* \left(\left(  S^{N\rho} \wedge (S^{\norm{J}-N\abs{\rho}} \wedge X'^{\norm{J}})_{h\Sigma_J} \right)_{h\Sigma_{(J)}}\right)/\mf{m}.
\end{align*}
It therefore suffices to show that the map $\Delta_J \wedge 1$ above is trivial.  We may assume that $N = 2M$ is even.  
To simplify notation let $Y$ denote the $\Sigma_{(J)}$-space
$$ Y := (S^{\norm{J}} \wedge X'^{\norm{J}})_{h\Sigma_J}. $$
We wish to show that the map 
$$ Y_{h\Sigma_{(J)}} \hookrightarrow 
\left(  S^{2M\rho^\perp} \wedge Y \right)_{h\Sigma_{(J)}} $$
is zero on mod $\mf{m}$ $E$-cohomology.

Using the fact that $E$ is complex orientable (and taking a fixed complex orientation), Lemma~\ref{lem:euler} implies that the effect on mod $\mf{m}$ $E$-cohomology of the above map, under the Thom isomorphism (\ref{eq:thom}), is multiplication by a power of the Euler class
$$ e_{2\rho^\perp}^M: \td{E}^*\left( Y_{h\Sigma_{(J)}} \right)/\mf{m} \rightarrow \td{E}^*\left( Y_{h\Sigma_{(J)}} \right)/\mf{m}. $$
Since $\rho^\perp$ is not zero dimensional, the colimit
$$ \varinjlim_M S^{2M\rho^\perp} $$
is non-equivariantly contractible, which implies that 
$$ \varinjlim_M {E\Sigma_{(J)}}_+ \wedge S^{2M\rho^\perp} $$
is equivariantly contractible, and therefore the localization is trivial:
$$ e_{2\rho^\perp}^{-1} E^*\left( B\Sigma_{(J)} \right)/\mf{m} = 0. $$
Since $G$ is finite, the $K_*$-vector space $E^*\left( B\Sigma_{(J)} \right)/\mf{m}$ is finite dimensional.  It follows (e.g. using Jordon canonical form after base change to $\bar{\FF}_p$) that the element $e_{2\rho^\perp}$ is nilpotent in $E^*(B\Sigma_{(J)})$, and therefore 
$$ e_{2\rho^\perp}^M : \td{E}^*\left( Y_{h\Sigma_{(J)}} \right)/\mf{m} \rightarrow \td{E}^*\left( Y_{h\Sigma_{(J)}} \right)/\mf{m} $$
is zero for $M$ sufficiently large.

The value of $M$ above depends on only on $k$ (which is either $2$ or $p$), the subgroup $\Sigma_{(J)} \le \Sigma_k$, and the representation $\rho^\perp$.  As there are only finitely many $k$ and subgroups, and associated representations in play, there is a global bound on the values of $M$ for the various subgroups showing up for different sequences $J$.  As such, $M$ can be taken sufficiently large to handle all of these different terms simultaneously.
\end{proof}

\subsection*{The norm equivalence of $H_\infty$-ring spectra}

\begin{lem}\label{lem:norm}
Suppose that $X$ is a space whose suspension spectrum is $K$-locally strongly dualizable.
Then the norm induces an equivalence
$$ \widehat{\PP}_{S_K}(S_K^X) = \prod_i [(S_K^{X^i})_{h\Sigma_i}]_K \xrightarrow[\simeq]{\prod N_{\Sigma_i}} \prod_i [S_K^{X^i}]^{h\Sigma_i} \simeq S_K^{\PP(X)^{\mr{add}}} $$
of $H_\infty$-$S_K$-algebras.
\end{lem}

Observe that since $K$-localization commutes with products when the factors involved are $E$-local (see \cite[Cor.~6.1.3]{BehrensDavis}) there is an equivalence
$$ 
\prod_i [(S_K^{X^i})_{h\Sigma_i}]_K \simeq \left[ \prod_i (S_K^{X^i})_{h\Sigma_i} \right]_K.
$$
Therefore, Lemma~\ref{lem:norm} follows from the following slightly more general observation.

\begin{lem}\label{lem:Rnorm}
The composite
$$ \widehat{\PP}_{R}(R^X) := \prod_i (R^{X^i})_{h\Sigma_i} \xrightarrow{\prod N_{\Sigma_i}} \prod_i (R^{X^i})^{h\Sigma_i} \simeq R^{\PP(X)^{\mr{add}}}  $$
is a map of $H_\infty$-$R$-algebras.
\end{lem}

The $H_\infty$-$R$-algebra structure of $\widehat{\PP}_R(R^X)$ has structure maps 
$$ \zeta_k: \widehat{\PP}_R(R^X)^{\wedge_R k}_{h\Sigma_k} \rightarrow \widehat{\PP}_R(R^X) $$
which are given by the composites
\begin{align*}
\zeta_k : 
\widehat{\PP}_R(R^X)^{\wedge_R k}_{h\Sigma_k} & = \left( \prod_i (R^{X^i})_{h\Sigma_i} \right)^{\wedge_R k}_{h\Sigma_k} \\
& \rightarrow \left( \prod_i \bigvee_{\substack{I = (i_1, \ldots, i_k) \\ \norm{I} = i}} (R^{X^{\norm{I}}})_{h\Sigma_I} \right)_{h\Sigma_k} \\
& \rightarrow \prod_i \left( \bigvee_{\substack{I = (i_1, \ldots, i_k) \\ \norm{I} = i}}  (R^{X^{\norm{I}}})_{h\Sigma_I} \right)_{h\Sigma_k}  \\
& \simeq \prod_i \bigvee_{[I] \in \mc{I}^k_i}  \left( (R^{X^{\norm{I}}})_{h\Sigma_I} \right)_{h\Sigma_{(I)}}  \\
& \simeq \prod_i \bigvee_{[I] \in \mc{I}^k_i}  (R^{X^{\norm{I}}})_{h\Sigma_{[I]}} \\
& \xrightarrow{\prod_i \sum_I \mr{Res}^{\Sigma_i}_{\Sigma_{[I]}}} \prod_i (R^{X^{i}})_{h\Sigma_i} \\
& = \widehat{\PP}_R(R^X). 
\end{align*}

Lemma~\ref{lem:Rnorm} therefore follows from the following lemma.

\begin{lem}
Fix a sequence $I=(i_1, \ldots, i_k)$ with $\norm{I} = i$.  Then the following diagram commutes.
$$
\xymatrix{
(R^{X^i})_{h\Sigma_{[I]}} \ar[r]^-{\mr{Res}^{\Sigma_i}_{\Sigma_{[I]}}} \ar[d]_{N_{\Sigma_{I}}} & (R^{X^i})_{h\Sigma_i} \ar[d]^{N_{\Sigma_i}} 
\\
(R^{X^i_{h\Sigma_I}})_{h\Sigma_{(I)}} \ar[r]_-{\xi^{\mr{add}}_{I,i}}  & R^{X^i_{h\Sigma_i}}
}
$$
\end{lem}

\begin{proof}
The construction of the norm as the adjoint to the equivariant transfer implies that the following diagram commutes.
$$
\xymatrix{
(R^{X^i})_{h\Sigma_{[I]}} \ar[rr]^-{\mr{Res}^{\Sigma_i}_{\Sigma_{[I]}}} \ar[d]_{N_{\Sigma_{I}}} 
\ar[dr]^{N_{\Sigma_{[I]}}}
&& (R^{X^i})_{h\Sigma_i} \ar[d]^{N_{\Sigma_i}} 
\\
(R^{X^i_{h\Sigma_I}})_{h\Sigma_{(I)}} \ar[r]_-{N_{\Sigma_{(I)}}}  & R^{X^i_{h\Sigma_{[I]}}} \ar[r]_{\mr{Tr}^{\Sigma_i}_{\Sigma_{[I]}}} &
R^{X^i_{h\Sigma_i}}
}
$$
Therefore it suffices to show the commutativity of the diagram below.
$$
\xymatrix{
(R^{X^i_{h\Sigma_I}})_{h\Sigma_{(I)}} \ar[r]^-{\xi^{\mr{add}}_{I,i}}  \ar[d]_{N_{\Sigma_{(I)}}}
& R^{X^i_{h\Sigma_i}}
\\
R^{X^i_{h\Sigma_{[I]}}} \ar[ur]_{\mr{Tr}^{\Sigma_i}_{\Sigma_{[I]}}}
}
$$
By adjointness this is equivalent to showing the commutativity of the diagram:
$$
\xymatrix{
(R^{X^i_{h\Sigma_I}})_{h\Sigma_{(I)}} \wedge X^i_{h\Sigma_i} \ar[r]^-{\td{\xi}^{\mr{add}}_{I,i}}  \ar[d]_{N_{\Sigma_{(I)}} \wedge 1}
& R 
\\
R^{X^i_{h\Sigma_{[I]}}} \wedge X^i_{h\Sigma_i} \ar[r]_{\mr{Tr}^{\Sigma_i}_{\Sigma_{[I]}}\wedge 1} &
R^{X^i_{h\Sigma_i}} \wedge X^i_{h\Sigma_i} \ar[u]_{\mr{ev}}
}
$$
By Lemma~\ref{lem:xiadd}, the commutativity of the above diagram is seen in the following commutative diagram.
\begin{equation}\label{eq:bigdiag}
{\scriptsize
\xymatrix@C-1em{
(R^{X^i_{h\Sigma_I}})_{h\Sigma_{(I)}} \wedge X^i_{h\Sigma_i} 
\ar@{=}[r] \ar[ddd]_{N_{\Sigma_{(I)}}\wedge 1} \ar[ddr]_{1 \wedge \mr{Tr}^{\Sigma_i}_{\Sigma_{[I]}}} & 
(R^{X^i_{h\Sigma_I}} \wedge X^i_{h\Sigma_i})_{h\Sigma_{(I)}} \ar[r]^{1 \wedge \mr{Tr}^{\Sigma_i}_{\Sigma_I}} 
 \ar[d]_{1 \wedge \mr{Tr}^{\Sigma_i}_{\Sigma_{[I]}}} &
(R^{X^i_{h\Sigma_I}} \wedge X^i_{h\Sigma_I})_{h\Sigma_{(I)}} \ar[r]^-{\mr{ev}} &
R_{h\Sigma_{(I)}} \ar[r]^-{\mr{Res}^1_{\Sigma_{(I)}}} & 
R
\\
&
(R^{X^i_{h\Sigma_I}} \wedge X^i_{h\Sigma_{[I]}})_{h\Sigma_{(I)}} 
 \ar[ur]|{1 \wedge \mr{Tr}^{\Sigma_{[I]}}_{\Sigma_{I}}} &
(R^{X^i_{h\Sigma_{[I]}}} \wedge X^i_{h\Sigma_I})_{h\Sigma_{(I)}} 
\ar[u]_{\mr{Res}^{\Sigma_{[I]}}_{\Sigma_I} \wedge 1} \ar@{=}[d] \ar@{}[urr]|{(2)} &
&
\\
&
(R^{X^i_{h\Sigma_I}})_{h\Sigma_{(I)}} \wedge X^i_{h\Sigma_{[I]}} 
\ar@{=}[u] \ar[r]^{N_{\Sigma_{(I)}}\wedge 1} \ar@{}[ur]|{(1)} &
R^{X^i_{h\Sigma_{[I]}}} \wedge X^i_{h\Sigma_{[I]}} \ar[uurr]_{\mr{ev}} &
&
\\
R^{X^i_{h\Sigma_{[I]}}} \wedge X^i_{h\Sigma_i} 
\ar[rrrr]_{\mr{Tr}^{\Sigma_{i}}_{\Sigma_{[I]}}\wedge 1} 
\ar[urr]_{1 \wedge \mr{Tr}^{\Sigma_i}_{\Sigma_{[I]}}} &
&
&
&
R^{X^i_{h\Sigma_{i}}} \wedge X^i_{h\Sigma_i} \ar[uuu]_{\mr{ev}}
}}
\end{equation}
With the exception of regions (1) and (2) in the above diagram, all of the faces of the diagram clearly commute.

The commutativity of $(1)$ is seen below, making use of Lemma~\ref{lem:twotransfer}.
$$
\xymatrix{
(R^{X^i_{h\Sigma_I}} \wedge (X^i_{h\Sigma_{I}})_{h\Sigma_{(I)}})_{h\Sigma_{(I)}} 
 \ar[rr]^-{1 \wedge \mr{Tr}^{\Sigma_{(I)}}_{1}}  \ar@{=}[dd] &
&
(R^{X^i_{h\Sigma_I}} \wedge X^i_{h\Sigma_I})_{h\Sigma_{(I)}} 
\\
&
((R^{X^i_{h\Sigma_{I}}})_{h\Sigma_{(I)}} \wedge X^i_{h\Sigma_{I}})_{h\Sigma_{(I)}}
\ar[ru]^-{\mr{Tr}^{\Sigma_{(I)}}_{1} \wedge 1} \ar[r]^{N_{\Sigma_{(I)}}\wedge 1} &
(R^{X^i_{h\Sigma_{[I]}}} \wedge X^i_{h\Sigma_I})_{h\Sigma_{(I)}} 
\ar[u]_{\mr{Res}^{\Sigma_{[I]}}_{\Sigma_I} \wedge 1} \ar@{=}[d]
\\
(R^{X^i_{h\Sigma_I}})_{h\Sigma_{(I)}} \wedge X^i_{h\Sigma_{[I]}} 
\ar[rr]_{N_{\Sigma_{(I)}}\wedge 1} \ar@{=}[ur] &
&
R^{X^i_{h\Sigma_{[I]}}} \wedge X^i_{h\Sigma_{[I]}}
}
$$

By adjointness, the commutativity of region (2) in Diagram~(\ref{eq:bigdiag}) is equivalent to the commutativity of the following diagram in $\mr{Ho}(\Sp_{\Sigma_{(I)}})$, which clearly commutes.
$$
\xymatrix@C+2em{
R^{X^i_{h\Sigma_I}} \wedge X^i_{h\Sigma_I} \ar[r]^-{\mr{ev}} & R 
\\
R^{X^i_{h\Sigma_{[I]}}} \wedge X^i_{h\Sigma_I} 
\ar[u]^{\mr{Res}^{\Sigma_{[I]}}_{\Sigma_I} \wedge 1} \ar[r]_-{1 \wedge \mr{Res}^{\Sigma_{[I]}}_{\Sigma_I} } &
R^{X^i_{h\Sigma_{[I]}}} \wedge X^i_{h\Sigma_{[I]}} \ar[u]_{\mr{ev}}
}
$$
\end{proof}

\bibliographystyle{amsalpha}
\nocite{*}
\bibliography{BKTAQ8}

\end{document}